\documentclass{jams-l}
\usepackage{eucal, mathrsfs}
\usepackage{amsmath,amssymb,amsthm,%amsfonts,
amscd,amsopn}
\usepackage{hyperref}
\usepackage{times}
\usepackage{xspace}
\usepackage{epsfig}
\usepackage{epsfig,epic} %eepic,
\usepackage{latexsym,color}
\usepackage[all]{xy}
\usepackage{comment} 
\usepackage{paralist}
\usepackage{stmaryrd}
\usepackage{enumitem}
\newcounter{are-there-sections}
\def\me{S\'ANDOR J KOV\'ACS\xspace}

\def\mythanksjoint{S\'andor Kov\'acs was supported in part by NSF Grants DMS-0554697
  and DMS-0856185, and the Craig McKibben and Sarah Merner Endowed Professorship in
  Mathematics.} 

\def\myaddress{University of Washington, Department of Mathematics, Box 354350,
Seattle, WA 98195-4350, USA}

\def\myemail{skovacs@uw.edu\xspace}

\def\myurladdr{http://www.math.washington.edu/$\sim$kovacs\xspace}

\DeclareMathAlphabet{\smallchanc}{OT1}{pzc}%
                                 {m}{it}
\DeclareFontFamily{OT1}{pzc}{}
\DeclareFontShape{OT1}{pzc}{m}{it}%
             {<-> s * [1.100] pzcmi7t}{}
\DeclareMathAlphabet{\mathchanc}{OT1}{pzc}%
                                 {m}{it}
\newcommand{\mcE}{\mathchanc{E}}

\newcommand{\mcH}{\mathchanc{H}}

\newcommand{\mcJ}{\mathchanc{J}}

\newcommand{\mcL}{\mathchanc{L}}
\newcommand{\myL}{\mcL}

\newcommand{\mcR}{\mathchanc{R}}
\newcommand{\myR}{\mcR\!} %{\bf R}}

\newcommand{\mcm}{\mathchanc{m}}

\newcommand{\mco}{\mathchanc{o}}

\newcommand{\mct}{\mathchanc{t}}

\newcommand{\mcx}{\mathchanc{x}}

\DeclareFontFamily{OMS}{rsfs}{\skewchar\font'60}
\DeclareFontShape{OMS}{rsfs}{m}{n}{<-5>rsfs5 <5-7>rsfs7 <7->rsfs10 }{}
\DeclareSymbolFont{rsfs}{OMS}{rsfs}{m}{n}
\DeclareSymbolFontAlphabet{\scr}{rsfs}

\newcommand{\sA}{\scr{A}}

\newcommand{\sF}{\scr{F}}

\newcommand{\sI}{\scr{I}}
\newcommand{\sJ}{\scr{J}}
\newcommand{\sK}{\scr{K}}
\newcommand{\sL}{\scr{L}}
\newcommand{\sM}{\scr{M}}
\newcommand{\sN}{\scr{N}}
\newcommand{\sO}{\scr{O}}

\newcommand{\sfA}{{\sf A}}
\newcommand{\sfB}{{\sf B}}
\newcommand{\sfC}{{\sf C}}
\newcommand{\sfD}{{\sf D}}

\newcommand{\sfJ}{{\sf J}}
\newcommand{\sfK}{{\sf K}}

\newcommand{\sfQ}{{\sf Q}}

\newcommand{\bC}{\mathbb{C}}

\newcommand{\bH}{\mathbb{H}}

\newcommand{\bN}{\mathbb{N}}

\newcommand{\bP}{\mathbb{P}}
\newcommand{\bQ}{\mathbb{Q}}

\newcommand{\bZ}{\mathbb{Z}}

\newcommand{\ff}{\mathfrak{f}}

\newcommand{\al}{\alpha}
\newcommand{\be}{\beta}
\newcommand{\ga}{\gamma}

\newcommand{\ol}{\overline}

\newcommand{\simq}[0]{\sim_{\q}}

\newcommand{\into}{\hookrightarrow}

\newcommand{\wt}{\widetilde}
\newcommand{\what}{\widehat}

\newcommand{\rdown}[1]{\lfloor{#1}\rfloor}

\newcommand{\leteq}{\colon\!\!\!=}

\DeclareMathOperator{\depth}{{depth}} 

\DeclareMathOperator{\diff}{Diff}

\DeclareMathOperator{\Ex}{Ex}

\DeclareMathOperator{\Ext}{Ext}

\newcommand{\sExt}[0]{{\mcE\mcx\mct}}

\DeclareMathOperator{\Hom}{Hom}
\newcommand{\sHom}[0]{{\mcH\mco\mcm}}    

\DeclareMathOperator{\id}{{id}}
\DeclareMathOperator{\im}{{im}}

\DeclareMathOperator{\red}{red}

\DeclareMathOperator{\Spec}{{Spec}}
\DeclareMathOperator{\supp}{{supp}}

\def\spec#1.#2.{{\mathbf S\mathbf p\mathbf e\mathbf c}_{#1}#2}
\def\proj#1.#2.{{\mathbf{Proj}}_{#1}\sum#2}
\def\ring#1.{\scr O_{#1}}
\def\map#1.#2.{#1 \to #2}

\def\longmap#1.#2.{#1 \longrightarrow #2}
\def\factor#1.#2.{\left. \raise 2pt\hbox{$#1$} \right/
        \hskip -2pt\raise -2pt\hbox{$#2$}}
\def\pe#1.{\mathbb P(#1)}
\def\pr#1.{\mathbb P^{#1}}
\newcommand{\bdot}{{{\,\begin{picture}(1,1)(-1,-2)\circle*{2}\end{picture}\ }}}
\newcommand{\clx}[1]{{#1}^{\bdot}}
\newcommand{\Om}{\underline{\Omega}}

\def\coh#1.#2.#3.{H^{#1}(#2,#3)}
\def\dimcoh#1.#2.#3.{h^{#1}(#2,#3)}
\def\hypcoh#1.#2.#3.{\mathbb H_{\vphantom{l}}^{#1}(#2,#3)}
\def\loccoh#1.#2.#3.#4.{H^{#1}_{#2}(#3,#4)}
\def\dimloccoh#1.#2.#3.#4.{h^{#1}_{#2}(#3,#4)}
\def\lochypcoh#1.#2.#3.#4.{\mathbb H^{#1}_{#2}(#3,#4)}
\def\ses#1.#2.#3.{0  \longrightarrow  #1   \longrightarrow 
 #2 \longrightarrow #3 \longrightarrow 0} 
\def\sesshort#1.#2.#3.{0
 \rightarrow #1 \rightarrow #2 \rightarrow #3 \rightarrow 0}
\def\dist#1.#2.#3.{  #1   \longrightarrow 
 #2 \longrightarrow #3 \stackrel{+1}{\longrightarrow} } % \tag{$\bigtriangleup$}}
\def\CDdist#1.#2.#3.{  #1   @>>>  #2  @>>>   #3 @>+1>> }  
\def\shortses#1.#2.#3.{0  \rightarrow  #1   \rightarrow 
 #2  \rightarrow   #3 \rightarrow  0}
\def\shortdist#1.#2.#3.{  #1   \rightarrow 
 #2  \rightarrow   #3 \stackrel{+1}{\rightarrow} }  % \tag{$\bigtriangleup$}}
\def\ddist#1.#2.#3.#4.#5.#6.{\CD
#1 @>>> #2 @>>> #3 @>+1>> \\
@VVV @VVV @VVV \\
#4 @>>> #5 @>>> #6 @>+1>> 
\endCD}
\def\ddistun#1.#2.#3.#4.#5.#6.{\CD
#1 @>>> #2 @>>> #3 @>+1>> \\
@. @VVV @VVV  \\
#4 @>>> #5 @>>> #6 @>+1>> 
\endCD}
\def\Iff#1#2#3{
\hfil\hbox{\hsize =#1
\vtop{\noin #2}
\hskip.5cm 
\lower.5\baselineskip\hbox{$\Leftrightarrow$}\hskip.5cm
\vtop{\noin #3}}\hfil\medskip}
\newcommand{\union}\cup
\newcommand{\intersect}\cap
\newcommand{\Union}\bigcup
\newcommand{\Intersect}\bigcap
\def\myoplus#1.#2.{\underset #1 \to {\overset #2 \to \oplus}}

\newcommand{\resto}{\big\vert_}

\def\ww#1.#2.{\curlywedge%\mskip-12mu\curlywedge
_{#1}^{#2}}
\def\wws#1.#2.{\scriptstyle\ww#1.#2.}

\def\fh#1.#2.{F^{#1} R^{#2}}
\def\epq#1.#2.#3.{E^{#1,#2}_{#3}}
\def\dd#1.#2.#3.{d^{\, #1,#2}_{#3}}
\def\kk#1.#2.#3.{K^{\, #1,#2}_{#3}}
\def\ki#1.#2.#3.{I^{\, #1,#2}_{#3}}
\def\ff#1.#2.{\mathfrak F_{#1}^{#2}}
\def\ffb#1.#2.{\boxed{\mathfrak F_{#1}^{#2}}}
\def\ffbx#1.#2.#3.{\boxed{\mathfrak F_{#1}^{#2}(#3)}}
\def\ffs#1.#2.{\scriptstyle\mathfrak F_{#1}^{#2}}
\def\ffsb#1.#2.{\boxed{\ffs#1.#2.}}
\def\ffsbx#1.#2.#3.{\boxed{\scriptstyle\mathfrak F_{#1}^{#2}(#3)}}
\def\fa#1.#2.{\mathfrak f_{#1}^{#2}}
\def\fab#1.#2.{\boxed{\mathfrak f_{#1}^{#2}}}
\def\ffa#1.#2.{\mathfrak F\mathfrak i\mathfrak l\mathfrak t_{#1}^{#2}}
\def\ffab#1.#2.{\boxed{\mathfrak F\mathfrak i\mathfrak l\mathfrak t_{#1}^{#2}}}
\def\rfi#1.#2.{R^{#1}\functor(#2\,)}

\def\al#1.#2.#3.{\alpha_{#1,#2}^{#3}}
\def\be#1.#2.{\beta_{#1,#2}}
\def\ga#1.#2.{\gamma_{#1,#2}}
\def\yy#1.#2.#3.#4.{y_{#1,#2}^{#3}(#4)}
\def\zz#1.#2.#3.#4.{z_{#1,#2}^{#3}(#4)}
\renewcommand{\o}[0]{{\scr O}} 
\newcommand{\q}[0]{{\mathbb Q}}
\newcommand{\qtq}[1]{\quad\mbox{#1}\quad}

\DeclareMathOperator{\nklt}{nklt}

\begin{document}

\definecolor{tomato}{RGB}{180,62,39}
\definecolor{forrest}{RGB}{81,133,49}    
\definecolor{lighttomato}{RGB}{253,65,65}
\definecolor{lightforrest}{RGB}{145,237,87}    
\definecolor{mygreen}{RGB}{40,104,69}    
\definecolor{mygreen2}{RGB}{3,149,39}    
\definecolor{darkolivegreen}{RGB}{102,118,75}
\definecolor{cranegreen}{RGB}{102,118,75}
\definecolor{mydarkblue}{RGB}{10,92,153}
\definecolor{myblue}{RGB}{57,222,186}
\definecolor{mauve}{RGB}{224, 176, 255}
\definecolor{fuchsia}{RGB}{255, 0, 255}
\definecolor{lavender}{RGB}{230, 230, 250}
\definecolor{pinkish}{RGB}{213,83,222}
\definecolor{colD}{RGB}{213,83,222}
\definecolor{defb}{RGB}{213,83,222}
\definecolor{goldenrod}{RGB}{225,115,69}

\makeatletter
\newenvironment{refmr}{}{}
\renewcommand{\labelenumi}{\hskip .5em(\thethm.\arabic{enumi})}
\setenumerate[1]{leftmargin=*,parsep=0em,itemsep=0.125em,topsep=0.125em}

\renewcommand\thesubsection{\thesection.\Alph{subsection}}
\renewcommand\subsection{
  \renewcommand{\sfdefault}{pag}
  \@startsection{subsection}%
  {2}{0pt}{-\baselineskip}{.2\baselineskip}{\raggedright
    \sffamily\itshape\small
  }}
\renewcommand\section{
  \renewcommand{\sfdefault}{phv}
  \@startsection{section} %
  {1}{0pt}{\baselineskip}{.2\baselineskip}{\centering
    \sffamily
    \scshape
    %\bfseries
}}
%\renewcommand\@cite[2]{{\rm [{#1\ifthenelse{\boolean{@tempswa}}{,\nolinebreak[3] #2}{}}]}}
%%%%%%%%%%
\newcounter{lastyear}\setcounter{lastyear}{\the\year}
\addtocounter{lastyear}{-1}
%%%%%%%%%%%
\newcommand\sideremark[1]{%
\normalmarginpar
\marginpar
[
\hskip .45in
\begin{minipage}{.75in}
  \begin{color}{red}
    \tiny #1
  \end{color}
\end{minipage}
]
{
\hskip -.075in
\begin{minipage}{.75in}
  \begin{color}{red}
    \tiny #1
  \end{color}
\end{minipage}
}}
\newcommand\rsideremark[1]{
\reversemarginpar
\marginpar
[
\hskip .45in
\begin{minipage}{.75in}
  \begin{color}{red}
    \tiny #1
  \end{color}
\end{minipage}
]
{
\hskip -.075in
\begin{minipage}{.75in}
  \begin{color}{red}
    \tiny #1
  \end{color}
\end{minipage}
}}

\newcommand\noin{\noindent}
\newcommand\hugeskip{\bigskip\bigskip\bigskip}
\newcommand\smc{\sc}
\newcommand\dsize{\displaystyle}
\newcommand\sh{\subheading}
\newcommand\nl{\newline}
%%%%%%%%%%%
%%%%%%%%% bibliography helpers
%%%%%%%%%%%
\newcommand\get{\input /home/kovacs/tex/latex/} %$ 
\newcommand\Get{\Input /home/kovacs/tex/latex/} %$ 
\newcommand\toappear{\rm (to appear)}
\newcommand\mycite[1]{[#1]}
\newcommand\myref[1]{(\ref{#1})}
\newcommand\myli{\hfill\newline\smallskip\noindent{$\bullet$}\quad}
\newcommand\vol[1]{{\bf #1}\ } 
\newcommand\yr[1]{\rm (#1)\ } 
%%%%%%%%%%%
%%%%%%%%% text abbreviations
%%%%%%%%%%%
\newcommand\cf{cf.\ \cite}
\newcommand\mycf{cf.\ \mycite}
\newcommand\te{there exist}
\newcommand\st{such that}
%%%%%%%%%%%
%%%%%%%%%%% Theorem Style: BOZONT
%%%%%%%%%%%
\newcommand\myskip{3pt}
\newtheoremstyle{bozont}{3pt}{3pt}%
     {\itshape}%         Body font
     {}%         Indent amount (empty = no indent, \parindent = para indent)
     {\bfseries}% Thm head font
     {.}%        Punctuation after thm head
     {.5em}%     Space after thm head (\newline = linebreak)
     {\thmname{#1}\thmnumber{ #2}\thmnote{ \rm #3}}%         Thm head spec
%%%%%%%%%%%%%%%%%%%%%%%%%%%%%%
\newtheoremstyle{bozont-sf}{3pt}{3pt}%
     {\itshape}%         Body font
     {}%         Indent amount (empty = no indent, \parindent = para indent)
     {\sffamily}% Thm head font
     {.}%        Punctuation after thm head
     {.5em}%     Space after thm head (\newline = linebreak)
     {\thmname{#1}\thmnumber{ #2}\thmnote{ \rm #3}}%         Thm head spec
%%%%%%%%%%%%%%%%%%%%%%%%%%%%%%
\newtheoremstyle{bozont-sc}{3pt}{3pt}%
     {\itshape}%         Body font
     {}%         Indent amount (empty = no indent, \parindent = para indent)
     {\scshape}% Thm head font
     {.}%        Punctuation after thm head
     {.5em}%     Space after thm head (\newline = linebreak)
     {\thmname{#1}\thmnumber{ #2}\thmnote{ \rm #3}}%         Thm head spec
%%%%%%%%%%%%%%%%%%%%%%%%%%%%%%
\newtheoremstyle{bozont-remark}{3pt}{3pt}%
     {}%         Body font
     {}%         Indent amount (empty = no indent, \parindent = para indent)
     {\scshape}% Thm head font
     {.}%        Punctuation after thm head
     {.5em}%     Space after thm head (\newline = linebreak)
     {\thmname{#1}\thmnumber{ #2}\thmnote{ \rm #3}}%         Thm head spec
%%%%%%%%%%%%%%%%%%%%%%%%%%%%%%
\newtheoremstyle{bozont-def}{3pt}{3pt}%
     {}%         Body font
     {}%         Indent amount (empty = no indent, \parindent = para indent)
     {\bfseries}% Thm head font
     {.}%        Punctuation after thm head
     {.5em}%     Space after thm head (\newline = linebreak)
     {\thmname{#1}\thmnumber{ #2}\thmnote{ \rm #3}}%         Thm head spec
%%%%%%%%%%%%%%%%%%%%%%%%%%%%%%
\newtheoremstyle{bozont-reverse}{3pt}{3pt}%
     {\itshape}%         Body font
     {}%         Indent amount (empty = no indent, \parindent = para indent)
     {\bfseries}% Thm head font
     {.}%        Punctuation after thm head
     {.5em}%     Space after thm head (\newline = linebreak)
     {\thmnumber{#2.}\thmname{ #1}\thmnote{ \rm #3}}%         Thm head spec
%%%%%%%%%%%%%%%%%%%%%%%%%%%%%%
\newtheoremstyle{bozont-reverse-sc}{3pt}{3pt}%
     {\itshape}%         Body font
     {}%         Indent amount (empty = no indent, \parindent = para indent)
     {\scshape}% Thm head font
     {.}%        Punctuation after thm head
     {.5em}%     Space after thm head (\newline = linebreak)
     {\thmnumber{#2.}\thmname{ #1}\thmnote{ \rm #3}}%         Thm head spec
%%%%%%%%%%%%%%%%%%%%%%%%%%%%%%
\newtheoremstyle{bozont-reverse-sf}{3pt}{3pt}%
     {\itshape}%         Body font
     {}%         Indent amount (empty = no indent, \parindent = para indent)
     {\sffamily}% Thm head font
     {.}%        Punctuation after thm head
     {.5em}%     Space after thm head (\newline = linebreak)
     {\thmnumber{#2.}\thmname{ #1}\thmnote{ \rm #3}}%         Thm head spec
%%%%%%%%%%%%%%%%%%%%%%%%%%%%%%
\newtheoremstyle{bozont-remark-reverse}{3pt}{3pt}%
     {}%         Body font
     {}%         Indent amount (empty = no indent, \parindent = para indent)
     {\sc}% Thm head font
     {.}%        Punctuation after thm head
     {.5em}%     Space after thm head (\newline = linebreak)
     {\thmnumber{#2.}\thmname{ #1}\thmnote{ \rm #3}}%         Thm head spec
%%%%%%%%%%%%%%%%%%%%%%%%%%%%%%
\newtheoremstyle{bozont-def-reverse}{3pt}{3pt}%
     {}%         Body font
     {}%         Indent amount (empty = no indent, \parindent = para indent)
     {\bfseries}% Thm head font
     {.}%        Punctuation after thm head
     {.5em}%     Space after thm head (\newline = linebreak)
     {\thmnumber{#2.}\thmname{ #1}\thmnote{ \rm #3}}%         Thm head spec
%%%%%%%%%%%%%%%%%%%%%%%%%%%%%%
\newtheoremstyle{bozont-def-newnum-bf}{3pt}{3pt}%
     {}%         Body font
     {}%         Indent amount (empty = no indent, \parindent = para indent)
     {\bfseries}% Thm head font
     {}%        Punctuation after thm head
     {.5em}%     Space after thm head (\newline = linebreak)
     {\thmnumber{#2.}\thmname{ #1}\thmnote{ \rm #3}}%         Thm head spec
%%%%%%%%%%%%%%%%%%%%%%%%%%%%%%
\newtheoremstyle{bozont-def-newnum-rm}{3pt}{3pt}%
     {}%         Body font
     {}%         Indent amount (empty = no indent, \parindent = para indent)
     {\rmfamily}% Thm head font
     {}%        Punctuation after thm head
     {.5em}%     Space after thm head (\newline = linebreak)
     {\thmnumber{#2.}\thmname{ #1}\thmnote{ \rm #3}}%         Thm head spec
%%%%%%%%%%%%%%%%%%%%%%%%%%%%%%
%%%%%%%%%%%%%%%%%%%%%%%%%%%%%% thms
%%%%%%%%%%%%%%%%%%%%%%%%%%%%%%
\theoremstyle{bozont}    
\ifnum \value{are-there-sections}=0 {%
  \newtheorem{proclaim}{Theorem}
} 
\else {%
  \newtheorem{proclaim}{Theorem}[section]
} 
\fi
%%%%%%%%%%%%%%%%%%%%%%%%%%%%%%
%%%%%%%%%%%%%%%%%%%%%%%%%%%%%%
\newtheorem{thm}[proclaim]{Theorem}
\newtheorem{mainthm}[proclaim]{Main Theorem}
\newtheorem{cor}[proclaim]{Corollary} 
\newtheorem{cors}[proclaim]{Corollaries} 
\newtheorem{lem}[proclaim]{Lemma} 
\newtheorem{prop}[proclaim]{Proposition} 
\newtheorem{conj}[proclaim]{Conjecture}
\newtheorem{subproclaim}[equation]{Theorem}
\newtheorem{subthm}[equation]{Theorem}
\newtheorem{subcor}[equation]{Corollary} 
\newtheorem{sublem}[equation]{Lemma} 
\newtheorem{subprop}[equation]{Proposition} 
\newtheorem{subconj}[equation]{Conjecture}
%%%%
\theoremstyle{bozont-sc}
\newtheorem{proclaim-special}[proclaim]{\specialthmname}
\newenvironment{proclaimspecial}[1]
     {\def\specialthmname{#1}\begin{proclaim-special}}
     {\end{proclaim-special}}
%%%%%%%%%%%%%%%%%%%%%%%%%%%%%%
\theoremstyle{bozont-remark}
\newtheorem{rem}[proclaim]{Remark}
\newtheorem{subrem}[equation]{Remark}
\newtheorem{notation}[proclaim]{Notation} 
\newtheorem{assume}[proclaim]{Assumptions} 
\newtheorem{obs}[proclaim]{Observation} 
\newtheorem{example}[proclaim]{Example} 
\newtheorem{examples}[proclaim]{Examples} 
\newtheorem{complem}[equation]{Complement}%!!!!!!!!!!!!!!!!!!!!!!
\newtheorem{const}[proclaim]{Construction}   %!!!!!!!!!!!!!!!!
\newtheorem{ex}[proclaim]{Exercise} 
\newtheorem{subnotation}[equation]{Notation} 
\newtheorem{subassume}[equation]{Assumptions} 
\newtheorem{subobs}[equation]{Observation} 
\newtheorem{subexample}[equation]{Example} 
\newtheorem{subex}[equation]{Exercise} 
\newtheorem{claim}[proclaim]{Claim} 
\newtheorem*{claim*}{Claim} 
\newtheorem{inclaim}[equation]{Claim} 
\newtheorem{subclaim}[equation]{Claim} 
\newtheorem{case}{Case} 
\newtheorem{subcase}{Subcase}   
\newtheorem{step}{Step}
\newtheorem{approach}{Approach}
\newtheorem{fact}{Fact}
\newtheorem{subsay}{}
%%\newcommand{\Subheading}[1]
%%{\def\SubHeadingName{#1}\begin{SubHeading}\end{SubHeading}}%  
\newtheorem*{SubHeading*}{\SubHeadingName}%
\newtheorem{SubHeading}[proclaim]{\SubHeadingName}
\newtheorem{sSubHeading}[equation]{\sSubHeadingName}
\newenvironment{demo}[1] {\def\SubHeadingName{#1}\begin{SubHeading}}
  {\end{SubHeading}}%  
\newenvironment{subdemo}[1]{\def\sSubHeadingName{#1}\begin{sSubHeading}}
  {\end{sSubHeading}} %
\newenvironment{demo-r}[1]{\def\SubHeadingName{#1}\begin{SubHeading-r}}
  {\end{SubHeading-r}}%
\newenvironment{subdemo-r}[1]{\def\sSubHeadingName{#1}\begin{sSubHeading-r}}
  {\end{sSubHeading-r}} %
\newenvironment{demo*}[1]{\def\SubHeadingName{#1}\begin{SubHeading*}}
  {\end{SubHeading*}}%
\newtheorem{SubHeadingEmpty}[proclaim]{\SubHeadingName}
\newtheorem{sSubHeadingEmpty}[equation]{\SubHeadingName}
\newenvironment{numero}{\def\SubHeadingName{}\begin{SubHeadingEmpty}}
  {\end{SubHeadingEmpty}}%  
\newenvironment{subnumero}{\def\SubHeadingName{}\begin{sSubHeadingEmpty}}
  {\end{sSubHeadingEmpty}}%  
\newtheorem{defini}[proclaim]{Definition}
\newtheorem{question}[proclaim]{Question}
\newtheorem{subquestion}[equation]{Question}
\newtheorem{crit}[proclaim]{Criterion}
\newtheorem{pitfall}[proclaim]{Pitfall}%!!!!!!!!!!!!!!!!!!!!!!
\newtheorem{addition}[proclaim]{Addition}%!!!!!!!!!!!!!!!!!!!!!!
\newtheorem{principle}[proclaim]{Principle} %!!!!!!!!!!!!!!!!!!!!!!
%%%%%%%%%%%%%%%%%%%%%%%%%%%%%%
%%%  these were at once \theoremstyle{bozont-def}
\newtheorem{condition}[proclaim]{Condition}
\newtheorem{say}[proclaim]{}
\newtheorem{exmp}[proclaim]{Example}
\newtheorem{hint}[proclaim]{Hint}
\newtheorem{exrc}[proclaim]{Exercise}
\newtheorem{prob}[proclaim]{Problem}
\newtheorem{ques}[proclaim]{Question}    %!!!!!!!!!!!!!!!!!!!!
\newtheorem{alg}[proclaim]{Algorithm}
\newtheorem{remk}[proclaim]{Remark}          
\newtheorem{note}[proclaim]{Note}            
\newtheorem{summ}[proclaim]{Summary}         
%%%%%%%%%%%%
\newtheorem{notationk}[proclaim]{Notation}   
\newtheorem{warning}[proclaim]{Warning}  
\newtheorem{defn-thm}[proclaim]{Definition--Theorem}  %!!!!!!!!!!!!!!!!!!!!!!!!
\newtheorem{convention}[proclaim]{Convention}  %!!!!!!!!!!!!!!!!!!!!!!!!!!!
%%%%%%%%%%%%%%%%%%%%%%%%%%%%%%
\newtheorem*{ack}{Acknowledgments}
%%%%%%%%%%%%%%%%%%%%%%%%%%%%%%
\theoremstyle{bozont-def}    
\newtheorem{defn}[proclaim]{Definition}
\newtheorem{subdefn}[equation]{Definition}
%%%%%%%%%%%%%%%%%%
\theoremstyle{bozont-reverse}    
\newtheorem{corr}[proclaim]{Corollary} 
\newtheorem{lemr}[proclaim]{Lemma} 
\newtheorem{propr}[proclaim]{Proposition} 
\newtheorem{conjr}[proclaim]{Conjecture}
%%%%
\theoremstyle{bozont-reverse-sc}
\newtheorem{proclaimr-special}[proclaim]{\specialthmname}
\newenvironment{proclaimspecialr}[1]%
{\def\specialthmname{#1}\begin{proclaimr-special}}%
{\end{proclaimr-special}}
%%%%%%%%%%%%%%%%%%%%%%%%%%%%%%
\theoremstyle{bozont-remark-reverse}
\newtheorem{remr}[proclaim]{Remark}
\newtheorem{subremr}[equation]{Remark}
\newtheorem{notationr}[proclaim]{Notation} 
\newtheorem{assumer}[proclaim]{Assumptions} 
\newtheorem{obsr}[proclaim]{Observation} 
\newtheorem{exampler}[proclaim]{Example} 
\newtheorem{exr}[proclaim]{Exercise} 
\newtheorem{claimr}[proclaim]{Claim} 
\newtheorem{inclaimr}[equation]{Claim} 
\newtheorem{SubHeading-r}[proclaim]{\SubHeadingName}
\newtheorem{sSubHeading-r}[equation]{\sSubHeadingName}
\newtheorem{SubHeadingr}[proclaim]{\SubHeadingName}
\newtheorem{sSubHeadingr}[equation]{\sSubHeadingName}
\newenvironment{demor}[1]{\def\SubHeadingName{#1}\begin{SubHeadingr}}{\end{SubHeadingr}}
\newtheorem{definir}[proclaim]{Definition}
%%%%%%%%%%%%%%%%%%%%%%%%%%%%%%
\theoremstyle{bozont-def-newnum-bf}    
\newtheorem{newnumbf}[proclaim]{}
\newtheorem{subnewnumbf}[equation]{}
%%%%%%%%%%%%%%%%%%%%%%%%%%%%%%
\theoremstyle{bozont-def-newnum-rm}    
\newtheorem{newnumrm}[proclaim]{}
\newtheorem{subnewnumrm}[equation]{}
%%%%%%%%%%%%%%%%%%%%%%%%%%%%%%
\theoremstyle{bozont-def-reverse}    
\newtheorem{defnr}[proclaim]{Definition}
\newtheorem{questionr}[proclaim]{Question}
\newtheorem{newnumspecial}[proclaim]{\specialnewnumname}
\newtheorem{subnewnumspecial}[equation]{\specialnewnumname}
\newenvironment{newnum}[1]{\def\specialnewnumname{#1}\begin{newnumspecial}}{\end{newnumspecial}}
%%%%%%%%%%%%%%%%%%
\numberwithin{equation}{proclaim}
\numberwithin{figure}{section} 
\newcommand\equinsect{\numberwithin{equation}{section}}
\newcommand\equinthm{\numberwithin{equation}{proclaim}}
\newcommand\figinthm{\numberwithin{figure}{proclaim}}
\newcommand\figinsect{\numberwithin{figure}{section}}

\newcommand{\num}{\arabic{section}.\arabic{proclaim}}
%%%%%%%%%%%%%%%%%
\newenvironment{pf}{\smallskip \noindent {\sc Proof. }}{\qed\smallskip}
%%%%%%%%%%%%%%%%%
%%%%%%%%%%% section/subsection
%%%%%%%%%%%%%%%%%
\newcommand\Section[1]
        { 
        \setcounter{firstsubsection}{0}
        \setcounter{proclaim}{0}
        \section{#1}
        \numberwithin{proclaim}{section}
        \numberwithin{equation}{proclaim}
        \numberwithin{figure}{proclaim} 
        \rm\noindent
        }
%%%%%%%%%%%%%%%%%
\newcounter{firstsubsection} % 
\setcounter{firstsubsection}{0} % 
\newcommand\Subsection[1] 
      { % 
        \ifnum \value{firstsubsection}=0 %
        \setcounter{subsection}{\value{proclaim}} %
        \setcounter{firstsubsection}{1} %
        \fi %
        \noin\subsection{#1} %
        \numberwithin{proclaim}{subsection} 
        \numberwithin{equation}{subsection} 
        \numberwithin{figure}{subsection} %
        }
%%%%%%%%%%%%%%%%%
\def\endSubsection
        {
        \setcounter{firstsubsection}{0}
        \setcounter{proclaim}{\value{subsection}}
        \numberwithin{proclaim}{section}
        \numberwithin{equation}{proclaim}
        \numberwithin{figure}{proclaim} 
        }
%%%%%%%%%%%%%%%%%
%%%%%%%%%%%%%%%%%
%%%%%%%%%%%%%%%%%
%%%%%%%%%%% non-theorem like environments
%%%%%%%%%%%%%%%%%
\newcounter{rosternumber}%\numberwithin{rosternumber}{proclaim}
%%%%%%%%%%%
\newenvironment{roster} {\begin{enumerate}
    \renewcommand{\labelenumi}{\refstepcounter{rosternumber}(\therosternumber)}}
  {\end{enumerate}}
%%%%%%%%%%%
\newenvironment{narrowroster}
{\vskip2pt\noindent\phantom{\quad}\begin{minipage}{4.75in}\begin{enumerate}
      \renewcommand{\labelenumi}{\refstepcounter{rosternumber}(\therosternumber)}}
    {\end{enumerate}\end{minipage}}
%%%%%%%%%%%
\newenvironment{Roster} {\begin{enumerate}\renewcommand{\labelenumi}{
      \refstepcounter{proclaim}(\arabic{section}.\arabic{subsection}.\arabic{proclaim})}}{\end{enumerate}} 
\newenvironment{bRoster} {\begin{enumerate}\renewcommand{\labelenumi}{
      \refstepcounter{proclaim}({\bf\theproclaim})}}{\end{enumerate}}
%%%%%%%%%%%
\def\rostitem{\refstepcounter{rosternumber}
        \item[(\therosternumber)]
        }
%%%%%%%%%%%
\newenvironment{enumerate-p}{
  \begin{enumerate}}
  {\setcounter{equation}{\value{enumi}}\end{enumerate}}
\newenvironment{enumerate-cont}{
  \begin{enumerate}
    {\setcounter{enumi}{\value{equation}}}}
  {\setcounter{equation}{\value{enumi}}
  \end{enumerate}}
%%%%%%%%%%%
\let\lenumi\labelenumi
\newcommand{\rmlabels}{\renewcommand{\labelenumi}{\rm \lenumi}}
\newcommand{\rmlabelsoff}{\renewcommand{\labelenumi}{\lenumi}}
%%%%%%%%%%%
\newenvironment{heading}{\begin{center} \sc}{\end{center}}
%%%%%%%%%%%
\newcommand\subheading[1]{\smallskip\noindent{{\bf #1.}\ }}
%%%%%%%%%%%%%%%%%%%%%%%%%%%%%%
\newlength{\swidth}
\setlength{\swidth}{\textwidth}
\addtolength{\swidth}{-,5\parindent}
\newenvironment{narrow}{
  \medskip\noindent\hfill\begin{minipage}{\swidth}}
  {\end{minipage}\medskip}
%%%%%%%%%%%%%%%%%%%%%%%%%%%%%%
\newcommand\nospace{\hskip-.5ex}
\makeatother
%%

%%%%%%%%%%%%%%%%%%%%%%%%%
\title{Log canonical singularities are Du Bois}
\author{J\'anos Koll\'ar and \me}
\date{\today}
\thanks{J\'anos Koll\'ar was supported in part by NSF Grant DMS-0758275. \\ 
\indent\mythanksjoint}
\address{JK: Department of Mathematics, Princeton University, Fine
  Hall, Washington Road, Princeton, NJ 08544-1000, USA} 
\email{kollar@math.princeton.edu}
\urladdr{http://www.math.princeton.edu/$\sim$kollar\xspace}
\address{SK: \myaddress}
\email{\myemail}
\urladdr{\myurladdr}
%\keywords{}
%
\maketitle
%\tableofcontents
%%%%%%%%%%%%%%%%%%%%%%%%%
\newcommand\structure{minimal qlc structure\xspace}
\newcommand\stratification{$f$-qlc stratification\xspace}
\newcommand\stratum{$f$-qlc stratum\xspace}
\newcommand\strata{$f$-qlc strata\xspace}
\newcommand\HX{\mcH_{X,f}}
%%%%%%%%%%%%%%%%%%%%%%%%%
\newcommand\uj[1]{%
  \begin{color}%
    {red}
  #1%
  \end{color}%
}
\newcommand\ujj[1]{%
  \begin{color}%
    {red}
  #1%
  \end{color}%
}
\newcommand\ujjj[1]{%
  \begin{color}%
    {red}
  #1%
  \end{color}%
}
\newcommand\ujv[1]{%
  \begin{color}%
    {red}
  #1%
  \end{color}%
}
\newcommand\ujvj[1]{%
  \begin{color}%
    {red}
  #1%
  \end{color}%
}
\renewcommand\uj{}
\renewcommand\ujj{}
\renewcommand\ujjj{}
\renewcommand\ujv{}
\renewcommand\ujvj{}

\section{Introduction}

\noindent
A recurring difficulty in the Minimal Model Program (MMP) is that while log terminal
singularities are quite well behaved (for instance, they are rational), \emph{log
  canonical} singularities are much more complicated; they need not even be
Cohen-Macaulay. The aim of this paper is to prove that, as conjectured in
\cite[1.13]{K+92}, log canonical singularities are \emph{Du Bois}. 
The concept of Du Bois singularities, abbreviated as \emph{DB}, was introduced by
Steenbrink in \cite{Steenbrink83} as a weakening of rationality.  It is not clear how
to define Du~Bois singularities in positive characteristic, so we work over a field
of characteristic $0$ throughout the paper.  The precise definition is rather
involved, see \eqref{def:DB}, but our main applications rely only on the following
consequence:
\begin{cor}
  Let $X$ be a proper scheme of finite type over $\bC$.  If $(X,\Delta)$ is log
  canonical for some $\bQ$-divisor $\Delta$, then the natural map
%  \sideremark{Perhaps works with $\bR$-divisors as well?}
  $$
  H^i(X^{\rm an},\bC)\to H^i(X^{\rm an},\sO_{X^{\rm an}})\cong  H^i(X,\sO_{X})
  $$
  is surjective for all $i$.
\end{cor}

Using \cite[Lemme 1]{MR0376678} this implies the following:

\begin{cor}%[(=Theorem~\ref{thm:coh-invariance})]
  \label{thm:coh-inv}
  Let $\phi:X\to B$ be a proper, flat morphism of complex varieties with $B$
  connected.  Assume that for all $b\in B$ there exists a $\bQ$-divisor $D_b$ on
  $X_b$ such that $(X_b,D_b)$ is log canonical.  Then $h^i(X_b, \sO_{X_b})$ is
  independent of $b\in B$ for all $i$.
\end{cor}

Notice that we do not require that the divisors $D_b$ form a family.

We also prove flatness of the cohomology sheaves of the relative dualizing complex of
a projective family of log canonical varieties \eqref{thm:flat}. Combining this
result with a Serre duality type criterion \eqref{thm:serre-duality} gives another
invariance property:

\begin{cor}%[(=Theorem~\ref{thm:cm-invariance})]
  \label{thm:cm-inv}
  Let $\phi:X\to B$ be a flat, projective morphism, $B$ connected. Assume that for all
  $b\in B$ there exists a $\bQ$-divisor $D_b$ on $X_b$ such that $(X_b,D_b)$ is log
  canonical.

  Then, if one fiber of $\phi$ is Cohen-Macaulay (resp.\ $S_k$ for some $k$), then
  all fibers are Cohen-Macaulay (resp.\ $S_k$).
\end{cor}

\begin{subrem}
  The $S_k$ case of this result answers a question posed to us by Valery Alexeev and
  Christopher Hacon.
\end{subrem}

For arbitrary flat, proper morphisms, the set of fibers that are Cohen-Macaulay
(resp.\ $S_k$) is open, but not necessarily closed. Thus the key point of
(\ref{thm:cm-inv}) is to show that this set is also closed.

The generalization of these results to the semi log canonical case turns out to be
%%change
not hard, but it needs some foundational work which will be presented elsewhere.  The
general case then implies that each connected component of the moduli space of stable
log varieties parametrizes either only Cohen-Macaulay or only non-Cohen-Macaulay
objects.

Let us first state a simplified version of our main theorem:

\begin{thm}\label{thm:main-weak}
  Let $({{X}},\Delta)$ be an lc pair.  Then $X$ is DB.
  More generally, let $W\subset X$ be a reduced, closed subscheme that is a union of
  log canonical centers of $({X},\Delta)$.  Then $W$ is DB.
\end{thm}

This settles the above mentioned conjecture \cite[1.13]{K+92}.  For earlier results
related to this conjecture see \cite[\S 12]{Kollar95s},
\cite{Kovacs99,Kovacs00c,Ishii85,Ishii87a,Ishii87b,MR2339829,KSS08,Schwede08}.

Actually, we prove a more general statement, which is better suited to working with
log canonical centers and allows for more general applications, but it might seem a
little technical for the first reading:

\begin{thm}\label{thm:main}
  Let $f:{{Y}}\to {{X}}$ be a proper \uj{surjective} morphism with connected fibers
  between normal varieties. Assume that there exists an effective $\bQ$-divisor on
  $Y$ such that $({{Y}},\Delta)$ is lc and $K_{{Y}}+\Delta\sim_{\q,f} 0$.  Then $X$
  is DB.

  More generally, let $W\subset Y$ be a reduced, closed subscheme that is a union of
  log canonical centers of $({{Y}},\Delta)$.  Then $f(W)\subset X$ is DB.
\end{thm}

%%change
There are three, more technical results that should be of independent interest. The
first is a quite flexible criterion for Du Bois singularities.

\begin{thm}\label{thm:db-criterion}
  Let ${f}: Y\to X$ be a proper %(birational)
  morphism between reduced schemes of finite type over $\bC$.
  %% change
  Let $W\subseteq X$ and $F\leteq {f}^{-1}(W)\subset Y$ be closed reduced subschemes
  with ideal sheaves $\sI_{W\subseteq X}$ and $\sI_{F\subseteq Y}$.  Assume that the
  natural map $\varrho$
  $$
  \xymatrix{ \sI_{W\subseteq X} \ar[r]_-\varrho & \myR{f}_*\sI_{F\subseteq Y}
    \ar@{-->}@/_1.5pc/[l]_{\varrho'} }
  $$
  admits a left inverse $\varrho'$, that is,
  $\varrho'\circ\varrho=\id_{\sI_{W\subseteq X}}$. Then if $Y,F$, and $W$ all have DB
  singularities, then so does $X$.
\end{thm}
\begin{subrem}
  Notice that we do not require ${f}$ to be birational. On the other hand the
  assumptions of the theorem and \cite[Thm~1]{Kovacs00b} imply that if $Y\setminus F$
  has rational singularities, e.g., if $Y$ is smooth, then $X\setminus W$ has
  rational singularities as well.
\end{subrem}

The second is a variant of the connectedness theorem \cite[17.4]{K+92} for not
necessarily birational morphisms.

\begin{thm}\label{lc.cent.thm}
  Let $f:{{Y}}\to {{X}}$ be a proper morphism with connected fibers between normal
  varieties. Assume that $({{Y}},\Delta)$ is lc and $K_{{Y}}+\Delta\sim_{\q,f} 0$.
  Let $Z_1, Z_2\subset Y$ be lc centers of $(Y,\Delta)$.
  Then, for every irreducible component $T\subset f(Z_1)\cap f(Z_2)$ there is an lc
  center $Z_T$ of $(Y,\Delta)$ such that $Z_T\subset Z_1$ and $f(Z_T)=T$.

  More precisely, let $g:Z_1\cap f^{-1}(T)\to S$ and $\pi:S\to T$ be the Stein
  factorization of $Z_1\cap f^{-1}(T)\to T$.  Then, for every irreducible component
  $S_i\subset S$ there is an lc center $Z_{S_i}$ of $(Y,\Delta)$ such that
  $g(Z_{S_i})=S_i$.
\end{thm}

The third is the flatness of the cohomology sheaves of the relative dualizing complex
of a DB morphism:

\begin{thm}\label{thm:flat}
  Let $\phi:X\to B$ be a flat projective morphism such that all fibers are Du~Bois.
  % cf.\ \eqref{def:DB}.
  Then the cohomology sheaves $h^{i}(\omega_\phi^\bdot)$ are flat over $B$, where
  $\omega_\phi^\bdot$ denotes the {relative dualizing complex} of $\phi$.
\end{thm}

\begin{demo}{Definitions and Notation}\label{demo:defs-and-not}
  Let $K$ be an algebraically closed field of characteristic $0$.  Unless otherwise
  stated, all objects are assumed to be defined over $K$, all schemes are assumed to
  be of finite type over $K$ and a morphism means a morphism between schemes of
  finite type over $K$. 

  \ujj{%
    If $\phi:Y\to Z$ is a birational morphism, then $\Ex(\phi)$ will denote the
    \emph{exceptional set} of $\phi$.  } For a closed subscheme $W\subseteq X$, the
  ideal sheaf of $W$ is denoted by $\sI_{W\subseteq X}$ or if no confusion is likely,
  then simply by $\sI_W$.  For a point $x\in X$, $\kappa(x)$ denotes the residue
  field of $\sO_{X,x}$.

  A \emph{pair} $(X,\Delta)$ consists of a variety $X$ and an effective $\bQ$-divisor
  $\Delta$ on $X$.  If $(X,\Delta)$ is a pair, then $\Delta$ is called a
  \emph{boundary} if $\rdown{(1-\varepsilon)\Delta}=0$ for all $0<\varepsilon<1$,
  i.e., the coefficients of all irreducible components of $\Delta$ are in the
  interval $[0,1]$.  For the definition of \emph{klt, dlt}, and \emph{lc} pairs see
  \cite{KM98} \ujj{and for the definition of the \emph{different}, $\diff$ see
    \cite[16.5]{K+92}}.
  Let $(X, \Delta)$ be a pair and 
  % , then we define a \emph{minimal dlt model} of  $(X,\Delta)$ the following way:
  % Let  
  $f^{\rm m}:X^{\rm m}\to X$ a \ujj{proper birational morphism such that $\Ex(f^{\rm
      m})$ is a} divisor $E=\sum a_iE_i$. Let $\Delta^{\rm m}\leteq (f^{\rm
    m})^{-1}_*\Delta + \sum_{a_i\leq -1}E_i$. Then $(X^{\rm m}, \Delta^{\rm m})$ is a
  \emph{minimal dlt model} of $(X,\Delta)$ if it is a dlt pair and the discrepancy of
  every $f^{\rm m}$-exceptional divisor is at most $-1$. Note that if $(X,\Delta)$ is
  lc with a minimal dlt model $(X^{\rm m}, \Delta^{\rm m})$, then $K_{X^{\rm
      m}}+\Delta^{\rm m} \sim_{\bQ} (f^{\rm m})^*(K_X+\Delta)$.
%%
%%if $(X',\Delta')$ is any pair with birational morphisms $g^{\rm m}:(X^{\rm m},
%%\Delta^{\rm m})\to (X',\Delta')$ and $g:(X', \Delta')\to (X,\Delta)$ such that
%%$g_*\Delta'=\Delta$ 

  For morphisms $\phi:X\to B$ and $\vartheta: T\to B$, the symbol $X_T$ will denote
  $X\times_B T$ and $\phi_T:X_T\to T$ the induced morphism.  In particular, for $b\in
  B$ we write $X_b = \phi^{-1}(b)$.
  Of course, by symmetry, we also have the notation $\vartheta_X:T_X\simeq X_T\to X$
  and if $\sF$ is an $\sO_X$-module, then $\sF_T$ will denote the $\sO_{X_T}$-module
  $\vartheta_X^*\sF$.

  For a morphism $\phi:X\to B$, the \emph{relative dualizing complex} of $\phi$ (if
  it exists) will be denoted by $\omega_\phi^\bdot$. \ujj{Recall that if $\phi$ is a
    projective morphism, then $\omega_\phi^\bdot=\phi^!\sO_B$.} In particular, for a
  (quasi-projective) scheme $X$, the \emph{dualizing complex} of $X$ will be denoted
  by $\omega_X^\bdot$.

%%%   \ujj{The \emph{embedded resolution} of $W$ is a proper birational morphism
%%%   $\pi:    \wt X\to X$ such that $\Ex(\pi)$ is a divisor and
%%%   $\pi^{-1}(W)+\Ex(\pi)$ is an    snc divisor.} 

  The symbol $\simeq$ will mean isomorphism in the appropriate category. In
  particular, between complexes considered as objects in a derived category it stands
  for a \emph{quasi-iso\-morph\-ism}.

  We will use the following notation: For a functor $\Phi$, $\myR\Phi$ denotes its
  derived functor on the (appropriate) derived category and $\myR^i\Phi:=
  h^i\circ\myR\Phi$ where $h^i(C^\bdot)$ is the cohomology of the complex $C^\bdot$
  at the $i^\text{th}$ term.  Similarly, $\bH^i_Z := h^i \circ \myR \Gamma_Z$ where
  $\Gamma_Z$ is the functor of cohomology with supports along a subscheme $Z$.
  Finally, $\sHom$ stands for the sheaf-Hom functor and $\sExt^i:=
  h^i\circ\myR\sHom$.
%
%  We will often use the notion that a morphism ${f}: \sfA\to \sfB$ in a derived
%  category \emph{has a left inverse}. This means that there exists a morphism $\psi:
%  \sfB\to \sfA$ in the same derived category such that $\psi\circ{f}:\sfA\to\sfA$
%  is the identity morphism of $\sfA$. I.e., $\psi$ is a \emph{left inverse} of
%  ${f}$. 
\end{demo}

\begin{demo-r}{DB singularities}\label{def:DB}
  Consider a complex algebraic variety $X$.  If $X$ is smooth and projective, its De
  Rham complex plays a fundamental role in understanding the geometry of $X$. When
  $X$ is singular, an analog of the De Rham complex, introduced by Du~Bois, plays a
  similar role.

  Let $X$ be a complex scheme of finite type. Based on Deligne's theory of mixed
  Hodge structures, Du~Bois defined a filtered complex of $\sO_X$-modules, denoted by
  $\Om_{X}^\bdot$, that agrees with the algebraic De Rham complex in a neighborhood
  of each smooth point and, like the De Rham complex on smooth varieties, its
  {analytization} provides a resolution of the sheaf of locally constant functions on
  $X$ \cite{DuBois81}.  \uj{Following H\'el\`ene Esnault's suggestion we will call
    $\Om_X^\bdot$ the \emph{Deligne-Du~Bois} complex.}

  Du~Bois observed that an important class of singularities are those for which
  $\Om_{X}^0$, the zeroth graded piece of the filtered complex $\Om_{X}^\bdot$, takes
  a particularly simple form. He pointed out that singularities satisfying this
  condition enjoy some of the nice Hodge-theoretic properties of smooth varieties
  cf.\ \eqref{thm:coh-invariance}.  These singularities were christened \emph{Du~Bois
    singularities} by Steenbrink \cite{Steenbrink83}. We will refer to them as
  \emph{DB singularities} and a variety with only DB singularities will be called
  \emph{DB}.

  The construction of the \uj{Deligne-}Du~Bois complex $\Om_X^\bdot$ is highly
  non-trivial, so we will not include it here. For a thorough treatment the
  interested reader should consult \cite[II.7.3]{MR2393625}. For alternative
  definitions, sufficient and equivalent criteria for DB singularities see
  \uj{\cite{Ishii85,Ishii87b,Kovacs99,Kovacs00c,MR2339829,KSS08}.}
\end{demo-r}

\begin{rem}\label{rem:db-is-sn}
  Recall that the seminormalization of $\sO_X$ is $h^0(\Om^0_X)$, the $0^{\text{th}}$
  cohomology sheaf of the complex $\Om^0_X$, and so $X$ is seminormal if and only if
  $\sO_X\simeq h^0(\Om^0_X)$ by \cite[5.2]{MR1741272} (cf.\ \cite[5.4.17]{Schwede06},
  \cite[4.8]{MR2339829}, and \uj{\cite[5.6]{MR2503989}}).  In particular, this
  implies that DB singularities are seminormal.
\end{rem}

\begin{ack}
  We would like to thank
%%change
  Valery Alexeev, H\'el\`ene Esnault, Osamu Fujino, Christopher Hacon, Max Lieblich
  and Karl Schwede for useful comments and discussions that we have benefited from.
  The otherwise unpublished Theorem~\ref{thm:dlt-models} was communicated to us by
  Christopher Hacon. We are grateful to Stefan Schr\"oer for letting us know about
  Example~\ref{ex:schroer}.  \ujv{We also thank the referee whose extremely careful
    reading helped us
    correct several errors and improve the presentation.}\\
\end{ack}

\section{A criterion for Du~Bois singularities}
%\input db-complex.tex

\begin{comment}
  
In order to prove our main theorem and more generally to prove that some
singularities are DB, it is useful to have a criterion that can be applied in various
situations. Our first result is such a criterion.

\begin{thm}\label{thm:db-criterion}
  Let ${f}: Y\to X$ be a proper %(birational)
  morphism between reduced schemes of finite type over $\bC$, $W\subseteq X$ an
  arbitrary subscheme, and $F\leteq {f}^{-1}(W)$, equipped with the induced
  reduced subscheme structure. Assume that the natural map $\varrho$
  $$
  \xymatrix{ \sI_{W\subseteq X} \ar[r]_-\varrho & \myR{f}_*\sI_{F\subseteq Y}
    \ar@{-->}@/_1.5pc/[l]_{\varrho'} }
  $$
  admits a left inverse $\varrho'$. Then if $Y,F$, and $W$ all have DB singularities,
  then so does $X$.
\end{thm}

\end{comment}

%\begin{subrem}\sideremark{More stuff here!}
%  Musing about why the ``obvious way'' in the case of a birational one would expect
%  to prove this does not work (or at least not so obvious).
%\end{subrem}

In order to prove \eqref{thm:db-criterion} we first need the following abstract
derived category statement.

\begin{lem}\label{lem:derived-split}
  Let $\sfA,\sfB,\sfC,\sfA',\sfB',\sfC'$ be objects in a derived category and assume
  that there exists a commutative diagram in which the rows form exact triangles:
  \begin{equation}
    \label{eq:1}
    \xymatrix{ %
      \sfA \ar[r]^{\phi} \ar[d]_\alpha & \sfB \ar[r]^\psi \ar[d]_\beta & \sfC
      \ar[r]^{\zeta} 
      \ar[d]_\gamma & \sfA[1] \ar[d]^{\alpha[1]}\\ 
      \sfA' \ar[r]_{\phi'} & \sfB' \ar[r]_{\psi'} & \sfC' \ar[r]_{\zeta'}  & \sfA'[1]
    }
  \end{equation}
  Then there exist an object $\sfD$, an exact triangle;
  \begin{equation}
    \label{eq:2}
    \xymatrix{ %
      \sfD \ar[r] & \sfB'\oplus \sfC \ar[r] %^-{(-\psi',\gamma)} 
      & \sfC' \ar[r]^{+1}
      &\text{\vphantom{A},}  % 
    }
  \end{equation}
  and a %natural
  map $\delta: \sfB\to \sfD$, such that if $\lambda$ denotes the composition $
  \xymatrix{ \lambda: \sfD\ar[r] & \sfB'\oplus \sfC \ar[r]^-{\ 0\oplus\id_{\sfC}\ }
    &\sfC}$, then $\lambda\circ\delta=\psi$ and
  \begin{enumerate-cont}\rmlabels
  \item $\alpha$ admits a left inverse if and only if $\delta$ admits one,
    $\delta':\sfD\to\sfB$ such that $\psi\circ\delta'=\lambda$, and
  \item $\alpha$ is an isomorphism if and only if $\delta$ is.
  \end{enumerate-cont}
\end{lem}

\begin{proof}
  Let $\eta: \sfB'\oplus \sfC \to \sfC'$ be the natural map induced by $-\psi'$ on
  $\sfB'$ and $\gamma$ on $\sfC$, and $\sfD$ the object that completes $\eta$ to an
  exact triangle as in (\ref{eq:2}).

  Next consider the following diagram:
  $$
  \xymatrix{ %
    && \sfA'[1] \ar[dddll]_{+1} \ar@{-->}[drr]^{\exists} & \\
    \sfC \ar[dd]_{+1} \ar@{-->}[urr]^{\exists\,\vartheta}  &&&& \sfD[1]
    \ar@{-->}[llll]_{+1}^{\exists}|!{[dd];[llu]}\hole |!{[llu];[lllldd]}\hole 
    \ar[dddll]_{+1}|!{[dd];[llu]}\hole|!{[ddllll];[dd]}\hole
    \\ 
    \\
    \sfB' \ar[rrrr]_{\psi'} \ar[rrd]_{(-\id_{\sfB'},0)} &&&& \sfC'
    \ar[lluuu]_(.45){\zeta'} \ar[uu]\\
    && \sfB'\oplus \sfC \ar[lluuu]_{0_{\sfB'}\oplus\id_{\sfC}}
    |!{[llu];[rru]}\hole|!{[uuuu];[ull]}\hole
    \ar[rru]_\eta \\  
  }
  $$
  The bottom triangle \uj{($B',C',B'\oplus C$)} is commutative with the maps
  indicated.  The triangles with one edge common with the bottom triangle are exact
  triangles with the obvious maps.  Then by the octahedral axiom, the maps in the top
  triangle denoted by the broken arrows exist and form an exact triangle.

  Observe that it follows that the induced map $\vartheta: \sfC\to \sfA'[1]$ agrees
  with $\zeta'\circ (\eta\resto{\sfC})=\zeta'\circ\gamma$ which in turn equals
  $\alpha[1]\circ\zeta$ by (\ref{eq:1}).

  Therefore one has the following commutative diagram where the rows form exact
  triangles:
  $$
  \xymatrix{ %
    \sfA \ar[r] \ar[d]_(0.35)\alpha & \sfB \ar[r]^\psi
    \ar@{-->}[d]_(0.35){\exists\delta} & \sfC
    \ar[r]^-{\zeta} \ar[d]_(0.35){\id_{\sfC}} & \sfA[1] \ar[d]_(0.35){\alpha[1]}\\
    \sfA' \ar[r] \ar@/_1pc/[u]_(0.35){\alpha'} & \sfD \ar[r]_\lambda
    \ar@/_1pc/@{-->}[u]_(0.35){\exists\delta'} & \sfC \ar@/_1pc/[u]_(0.35){\id_{\sfC}} 
    \ar[r]_-{\vartheta} & \sfA'[1],\!\! \ar@/_1pc/[u]_(0.35){\alpha'[1]}  \\ 
  }
  $$
  and as $\vartheta=\alpha[1]\circ\zeta$ it follows that a $\delta$ exists that makes
  the diagram commutative. Now, if $\alpha$ admits a left inverse $\alpha':\sfA'\to
  \sfA$, then $\alpha'[1]\circ\vartheta= \alpha'[1]\circ\alpha[1]\circ\zeta=
  \zeta=\zeta\circ\id_{\sfC}$, and hence $\delta$ admits a left inverse,
  $\delta':\sfD\to \sfB$ and clearly $\psi\circ\delta'=\lambda$. The converse is even
  simpler: if $\psi\circ\delta'=\lambda$, then $\alpha'$ exists and it must be a left
  inverse. Finally, it is obvious from the diagram that $\alpha$ is an isomorphism if
  and only if $\delta$ is.
\end{proof}

We are now ready to prove our DB criterion.

%\rsideremark{Numbering changed!}
\begin{newnum}{\it Proof of \eqref{thm:db-criterion}}
  % \sideremark{REMOVED THE SENTENCE: First, observe that by the nature of the
  %   statement we may assume that $W\neq X$.} %
  Consider the following commutative diagram with exact rows:
  \begin{equation*}
    % \label{eq:4}
    \xymatrix{
      \sI_{W\subseteq X} \ar[r] \ar[d]_{\varrho} & \sO_X \ar[r]  \ar[d] & \sO_W
      \ar[r]^-{+1}  \ar[d] & \\ 
      \myR{f}_*\sI_{F\subseteq Y} \ar[r] \ar@{-->}@/_1pc/[u]_{\varrho'} & \myR{f}_*\sO_Y \ar[r]
      & \myR{f}_*\sO_F \ar[r]^-{+1}   & \\ 
    }
  \end{equation*}
  It follows by Lemma~\ref{lem:derived-split} that there exists an object $\sfQ$,
  an exact triangle in the derived category of $\sO_X$-modules,
  \begin{equation}
    \label{eq:5}
    \xymatrix{
      \sfQ \ar[r] & \myR{f}_*\sO_Y \oplus \sO_W \ar[r] & \myR{f}_*\sO_F
      \ar[r]^-{+1} &
    },
  \end{equation}
  and a %natural
  map $\vartheta:\sO_X\to \sfQ$ that admits a left inverse, $\vartheta':\sfQ\to
  \sO_X$.
  
  Now consider a similar commutative diagram with exact rows:
  $$
  \xymatrix{ \sfJ \ar[r] \ar@{-->}[d]_\psi & \Om^0_X \ar[r] \ar[d]_\mu & \Om^0_W
    \ar[r]^-{+1}  \ar[d]_\nu & \\
    \myR{f}_*\sfK \ar[r] & \myR{f}_*\Om^0_Y \ar[r] & \myR{f}_*\Om^0_F \ar[r]^-{+1} &
    .  }
  $$
  Here $\sfJ$ and $\sfK$ represent the appropriate objects in the appropriate derived
  categories that make the rows exact triangles. The vertical maps $\mu$ and $\nu$
  exist and form a commutative square because of the basic properties of the
  \uj{Deligne-}Du~Bois complex and their existence and compatibility imply the
  existence of the map $\psi$ by the basic properties of derived categories.

  It follows, again, by Lemma~\ref{lem:derived-split} that there exists an object
  $\sfD$, an exact triangle in the derived category of $\sO_X$-modules,
  \begin{equation}
    \label{eq:6}
    \xymatrix{
      \sfD \ar[r] & \myR{f}_*\Om^0_Y \oplus \Om^0_W \ar[r] & \myR{f}_*\Om^0_F
      \ar[r]^-{+1} &  
    },
  \end{equation}
  and a %natural
  map $\delta:\Om^0_X\to \sfD$.

  %%% The proof of Lemma~\ref{lem:derived-split} shows that
  \uj{By the basic properties of exact triangles,} the natural transformation $\Xi:
  \sO \to \Om^0$ induces compatible maps between the exact triangles of (\ref{eq:5})
  and (\ref{eq:6}). \uj{We would also like to show that these maps are compatible
    with the maps $\vartheta$, and $\delta$ obtained from
    Lemma~\ref{lem:derived-split}}:
  $$
  \xymatrix{%
    \sO_X \ar[dr]_\vartheta \ar[d]_\lambda\\
    \Om^0_X \ar[dr]_\delta & \sfQ \ar[r] \ar@{->}@/_1pc/[ul]_{\vartheta'}
    \ar@{-->}[d]_-\xi & \myR\pi_*\sO_{Y}\oplus \sO_W \ar[r]
    \ar[d]_\eta & \myR\pi_*\sO_F    \ar[r]^-{+1} \ar[d]_\zeta & \\
    & \sfD \ar[r] & \myR\pi_*\Om^0_{Y}\oplus \Om^0_W \ar[r] & \myR\pi_*\Om^0_F
    \ar[r]^-{+1} &.  }
  $$

  \uj{
    \begin{subclaim}\label{claim:diagram-commutes}
      Under the assumptions of the theorem the above diagram is commutative.
    \end{subclaim}
    
    \begin{proof}
      First observe that if $Y, F$, and $W$ are all DB, then $\eta$ and $\zeta$ are
      isomorphisms. Then it follows that $\xi$ is an isomorphism as well. Next
      consider the $0^\text{th}$ cohomology sheaves of all the complexes in the
      diagram. From the long exact sequence of cohomology induced by exact triangles
      we obtain the following diagram:
      $$
      \xymatrix{%
        \sO_X \ar[dr]_{h^0(\vartheta)} \ar[d]_{h^0(\lambda)}\\
        h^0(\Om^0_X) \ar[dr]_{h^0(\delta)} & h^0(\sfQ) \ar@{^{(}->}[r]^-\nu
        \ar@{->}@/_1pc/[ul]_{h^0(\vartheta')} \ar[d]^{h^0(\xi)}_\simeq &
        h^0(\myR\pi_*\sO_{Y}\oplus \sO_W) \simeq \pi_*\sO_{Y}\oplus \sO_W
        \ar@<-5ex>[d]^{h^0(\eta)}_\simeq \\ 
        & h^0(\sfD) \ar@{^{(}->}[r]_-\mu & h^0(\myR\pi_*\Om^0_{Y}\oplus \Om^0_W) \simeq
        \pi_*\sO_{Y}\oplus \sO_W  }
      $$
      From the commutativity of the exact triangles we obtain that
      $$
      h^0(\eta)\circ\nu=\mu\circ h^0(\xi). 
      $$
      Furthermore, the functoriality of the maps ${h^0(\lambda)}$ and $h^0(\eta)$ (they are
      induced by $\Xi$) implies that we also have
      $$
      h^0(\eta)\circ\nu\circ h^0(\vartheta) =\mu\circ h^0(\delta)\circ {h^0(\lambda)}.
      $$
      Then it follows that 
      $$
      \mu\circ h^0(\xi)\circ h^0(\vartheta) =\mu\circ h^0(\delta)\circ {h^0(\lambda)}.
      $$
      Observe that $\mu$ is injective since $h^{-1}(\myR{\pi}_*\Om^0_F)=0$ and hence
      $$
      h^0(\xi)\circ h^0(\vartheta) = h^0(\delta)\circ {h^0(\lambda)}.
      $$
      Finally observe that ${h^0(\lambda)}$ determines the entire map
      $\lambda:\sO_X\to \Om^0_X$ by \eqref{subclaim:h-nought} and so we obtain that
      $\xi\circ\vartheta=\delta\circ\lambda$ as desired.
    \end{proof} 

    \begin{subclaim}\label{subclaim:h-nought}
      Let $\sfA,\sfB$ objects in a derived category such that $h^{i}(\sfA)=0$ for
      $i\neq 0$ and $h^j(\sfB)=0$ for $j<0$. Then any morphism $\alpha:\sfA\to\sfB$
      is uniquely determined by $h^0(\alpha)$.
    \end{subclaim} 
    
    \begin{proof}
      By the assumption, the morphism $\alpha:\sfA\to \sfB$ can be represented by a
      morphism of complexes $\wt\alpha:\wt\sfA\to\wt\sfB$, where $\sfA\simeq \wt\sfA$
      such that $\wt\sfA^0= h^0(\sfA)$ and $\wt\sfA^i=0$ for all $i\neq 0$, and
      $\sfB\simeq \wt\sfB$ such that $h^0(\wt\sfB)\subseteq \wt\sfB^0$. However
      $\wt\alpha$ has only one non-zero term, $h^0(\alpha)$. This proves the claim.
    \end{proof}
  }

  \noindent
  As $\xi$ is an isomorphism, we obtain a map, 
  \ujjj{%
    $$
    \lambda'=\vartheta'\circ\xi^{-1}\circ \delta: \Om^0_X\to \sO_X.
    $$
    By \eqref{claim:diagram-commutes} $\xi\circ\vartheta=\delta\circ\lambda$, so it
    follows that
    $$
    \lambda'\circ\lambda=\vartheta'\circ\xi^{-1}\circ \delta\circ\lambda=
    \vartheta'\circ\xi^{-1}\circ\xi\circ\vartheta=
    \vartheta'\circ\vartheta=\id_{\sO_X}.
    $$
    The last equality follows from the choice of $\vartheta'$.%
  }
  % Next we want to prove that $\lambda'$ is a left inverse to the natural map
  % $\lambda:\sO_X\to\Om^0_X$, i.e., that $\lambda'\circ\lambda=\id_{\sO_X}$.
  %    % 
  % However, this was essentially proved in \eqref{claim:diagram-commutes}.  Using
  % the notation from \eqref{claim:diagram-commutes}, observe that the compostion
  % $\nu\circ\vartheta$,
  % 
  % restrict all the maps to a \uj{dense} open set $U\subseteq X$ such that $U$ is
  % smooth and $U\cap W=\emptyset$.  , and ${f}\resto{{f}^{-1}(U)}$ is a smooth
  % morphism.  Such a $U$ exists \uj{by (\ref{eq:13})}.  It follows from the
  % construction that
  % $$
  % \xymatrix{ \sO_ X \ar[r]_-{\vartheta} \ar@/^1pc/[rr] & \sfQ \ar[r]_-{\nu} &
  %   \myR{f}_*\sO_Y }
  % $$
  % is the ``usual'' natural map and similarly
  % $$
  % \xymatrix{ \sO_X\ar[r]_-{\lambda|_{_X}}^-\simeq & \Om^0_X\resto
  %   U\ar[r]_-{\delta|_{_X}} & \sfD\resto U\simeq \myR{f}_*\Om^0_Y\resto U
  %   \ar[r]_-{\xi^{-1}|_{_X}}^-\simeq & \myR{f}_*\sO_Y\resto U }
  % $$
  % is the same ``usual'' natural map. One simple way to see this is to follow around
  % the image of the constant section $1\in \sO_X$. Therefore
  % $\lambda'\circ\lambda\resto U=\vartheta'\circ\vartheta\resto U=\id_{\sO_X}$.
  % However, as $X$ is reduced this implies that then
  \ujjj{%
    Therefore $\lambda'$ is a left inverse to $\lambda$ and so} the statement follows
  from \cite[Thm.~2.3]{Kovacs99}.  \qed
%
%  For the seminormal case of the statement first recall that the seminormalization of
%  $\sO_X$ is $h^0(\Om^o_X)$, the $0^{\text{th}}$ cohomology sheaf of the complex
%  $\Om^o_X$, and so $X$ is seminormal if and only $\sO_X\simeq h^0(\Om^o_X)$ by
%  \cite[5.2]{MR1741272} (cf.\ \cite[5.4.17]{Schwede06} and \cite[4.8]{MR2339829}).
%  Repeating the above proof with $\Om^o$ replaced by $h^0(\Om^o)$ everywhere gives
%  that the natural map $\sO_X\to h^0(\Om^o_X)$ has a left inverse.  However, both of
%  these sheaves are torsion free and the map is an isomorphism on an open set and
%  hence it must be an isomorphism everywhere.
\end{newnum}

We have a similar statement for seminormality. The proof is however much more
elementary.

\begin{prop}
  Let ${f}: Y\to X$ be a proper %(birational)
  morphism between reduced schemes of finite type over $\bC$, $W\subseteq X$ a closed
  reduced subscheme, and $F\leteq {f}^{-1}(W)$, equipped with the induced reduced
  subscheme structure. Assume that the natural map $\sI_{W\subseteq X} \to
  {f}_*\sI_{F\subseteq Y}$ is an isomorphism. Then if $Y$ and $W$ are seminormal,
  then so is $X$.
\end{prop}

\begin{proof}
  First, not yet assuming that $W$ or $Y$ are seminormal, let $\sO^{sn}_W$ be the
  seminormalization of $\sO_W$ in ${f}_*\sO_F$ and $\sO_X^{sn}$ the seminormalization
  of $\sO_X$ in ${f}_*\sO_Y$.  It follows from the assumption $\sI_{W\subseteq
    X}\simeq {f}_*\sI_{F\subseteq Y}$ that $\sO_W^{sn}\simeq
  \factor\sO_X^{sn}.\sI_{W\subseteq X}.$. Now if $W$ is seminormal (in fact it is
  enough if $\sO_W$ is seminormal in ${f}_*\sO_F$), then this implies that
  $\factor\sO_X.\sI_{W\subseteq X}. \simeq \sO_W \simeq \sO_W^{sn}\simeq
  \factor\sO_X^{sn}.\sI_{W\subseteq X}.$ and hence $\sO_X\simeq \sO^{sn}_X$.  Finally
  if $Y$ is seminormal, then this implies that so is $X$.
\end{proof}

\begin{cor}\label{cor:finite-cover-is-ok}\label{cor:finite-cover-to-normalization}
  Let $g:X'\to X$ be a finite surjective morphism between normal varieties. Let
  $Z\subseteq X$ be a reduced (not necessarily normal) subscheme and assume that
  $Z':= g^{-1}(Z)_{\red}$ is DB.  Then so is $Z$.
%
%  In particular, if $X'$ is DB, then so is $X$.
\end{cor}

\begin{rem}
  The special case of this statement when $Z=X$ was proved in
  \cite[12.8.2]{Kollar95s} for $X$ projective and in \cite[2.5]{Kovacs99} in general.
\end{rem}

\begin{proof}[Proof \uj{ of \eqref{cor:finite-cover-is-ok}}]
  \ujjj{%
    Let $\tau:g_*\sO_{X'}\to \sO_X$ denote the normalized trace map of $X'\to X$ and
%    $\sJ\subseteq g_*\sO_{X'}$ the integral closure of $\sI_{Z\subseteq X}$ in
%    $g_*\sO_{X'}$.  By \cite[5.14]{MR0242802} 
    $\sJ=\sqrt{\sI_{Z\subseteq X}\cdot g_*\sO_{X'}}$. Then it follows from
    \cite[5.14, 5.15]{MR0242802} and the fact that $Z$ is reduced, that
    \begin{equation*}
      % \label{eq:15}
      \tau(\sJ)\subseteq \sqrt{\sI_{Z\subseteq X}}={\sI_{Z\subseteq X}}.
    \end{equation*}
    Therefore, $\tau$ gives a splitting of} $\o_Z\to f_*\o_{Z'}$. The rest follows
  from \eqref{thm:db-criterion} applied to $Z$ with $W=\emptyset$ \uj{or directly by
    \cite[2.4]{Kovacs99}}.
\end{proof}

\section{Dlt models and %splitting of
  twisted higher direct images of dualizing sheaves}

We will frequently use the following statement in order to pass from an lc pair to
its dlt model.  Please recall the definition of a \emph{boundary} and a \emph{minimal
  dlt model} from \eqref{demo:defs-and-not}.

\begin{thm}[(Hacon)] \label{thm:dlt-models} 
% \label{thm:dlt-models-gen} 
  Let $(X,\Delta)$ be a pair such that $X$ is quasi-projective, $\Delta$ is a
  boundary, and $K_X+\Delta$ is a $\bQ$-Cartier divisor.  Then $(X,\Delta)$ admits a
  $\bQ$-factorial minimal dlt model $f^{\rm m}:(X^{\rm m}, \Delta^{\rm m})\to
  (X,\Delta)$.
\end{thm}

\begin{comment}
  
\noindent
During the proof we will need the following virtual strengthening of
\cite[3.39]{KM98}:

\begin{lem}\label{lem:negative-nef-is-effective}
  Let $h:Z\to Y$ be a proper birational morphism between normal varieties and
  $\Gamma$ a $\bQ$-Cartier $\bQ$-divisor on $Z$. Assume that $-\Gamma$ is generically
  $h$-nef in the following sense: There exists a closed subset $W\subset Z$ of
  codimension at least $2$ with the property that $-\Gamma\cdot \gamma\geq 0$ for any
  $y\in Y$ and any curve $\gamma\subset h^{-1}(y)$ such that $\gamma\not\subset W$.

  Then $\Gamma$ is effective if and only if $h_*\Gamma$ is.
\end{lem}

\begin{proof}
  The proof of \cite[3.39]{KM98} works verbatim to prove this statement.
\end{proof}

\begin{proof}[Proof of \ref{thm:dlt-models}]

\end{comment}

\begin{proof}
  Let $f:X'\to (X,\Delta)$ be a log resolution that is a composite of blow-ups of
  centers of codimension at least $2$.  Note that then there exists an effective
  $f$-exceptional divisor $C$ such that $-C$ is $f$-ample.
  Let $\Delta'$ be a divisor such that $\Delta'- f^{-1}_*\Delta$ is $f$-exceptional
  and that 
  $$
  K_{X'}+\Delta'\sim_{\bQ} f^*(K_X+\Delta).
  $$ 
  Let
  $\Delta_{<1}=\Delta-\rdown\Delta$ denote the part of $\Delta$ with coefficients
  strictly less than $1$, and write %$\Delta'$, % in the following way,
  $$
    \Delta' = f^{-1}_*\Delta_{<1} + 
    E^+ + F - B,
  $$
  where % $A$ is the sum of all $f$-exceptional divisors with discrepancy $<-1$,
  $E^+$ denotes the sum of all (not necessarily exceptional) divisors with
  discrepancy $\leq -1$, $F$ the sum of all $f$-exceptional divisors with discrepancy
  $>-1$ and $\leq 0$, and $B$ the sum of all $f$-exceptional divisors with
  discrepancy $>0$.  Let $E\leteq \red E^+$ and notice that all of $E,F$, and $B$ are
  effective, $E^+-E,F$, and $B$ are $f$-exceptional, %pairwise without common components, 
  while $f^{-1}_*\Delta_{<1}+E+F$ contains $f^{-1}_*\Delta$.

  Let $H$ be sufficiently ample on $X$.  Then for all $\varepsilon, \mu, \nu\in \bQ$,
%  \addtocounter{proclaim}{-1}
  \begin{equation}
    \label{eq:9}
    E+(1+\nu)F+\mu (-C+ f^*H)
    =
    (1-\varepsilon\mu)E+(1+\nu)F+\mu(\varepsilon E-C+f^*H).
  \end{equation}
%  \addtocounter{proclaim}{1}
  If $0<\varepsilon\ll 1$, then both $-C+ f^*H$ and $\varepsilon E-C+ f^*H$ are
  ample, hence $\bQ$-linearly equivalent to divisors $H_1$ and $H_2$ such that
  $\Delta'+H_1+H_2$ has snc support.  If $0<\mu<1$ and $0<\nu\ll 1$, then, by the
  definition of $E$ and $F$,
  $$
  (X',f^{-1}_*\Delta_{<1} + (1-\varepsilon\mu)E+(1+\nu) F+ \mu H_2)
  $$ 
  is klt and hence by \cite{BCHMc06} it has a ($\bQ$-factorial) minimal model
  $f^{\text m}:(X^{\text m},\Delta^{\text m}_{\varepsilon, \mu,\nu}) \to X$. By
  (\ref{eq:9}) this is also a minimal model of the pair $(X',f^{-1}_*\Delta_{<1} +
  E+(1+\nu)F+ \mu H_1)$, which is therefore dlt.  Let $\Delta^{\text m}$ denote the
  birational transform of $f^{-1}_*\Delta_{<1} + E+F$ on $X^{\text m}$.  Then we
  obtain that $(X^{\text m}, \Delta^{\text m})$ is dlt.

  For a divisor $G\subset X'$ (e.g., $E,F,C,H_i$) appearing above (other than
  $\Delta$) let $G^{\text m}$ denote its birational transform on $X^{\text m}$.
  By construction
  $$
  N\leteq K_{X^{\text m}}+ \Delta^{\text m}+ \nu F^{\text m} +\mu  H^{\text m}_1 
  \sim_{\bQ}
  K_{X^{\text m}}+   \Delta^{\text m}_{\varepsilon, \mu,\nu}
  $$
  is $f^{\text m}$-nef and
  $$
  T\leteq K_{X^{\text m}}+ \Delta^{\text m}+ (E^+-E)^{\text m} -B^{\text m}
  \sim_{\bQ} (f^{\text m})^*(K_{X}+\Delta)
  $$ 
  is $\bQ$-linearly $f^{\text m}$-trivial.
  %
  % Then $B^{\text m}$ is $\bQ$-linearly $f^{\text m}$-equivalent to
  % $N+(E^+-E)^{\text m}-\nu F^{\text m}-\mu H_1^{\text m}$.
  %
  Let
  $$
  D^{\text m} := % \Delta_{\varepsilon,\mu,\nu}^{\text m}- E^{\text m} -F^{\text m} +
  \mu C^{\text m} +(E^+-E)^{\text m} -\nu F^{\text m} - B^{\text m}.
  %\sim_{f^{\text      m},\q} T-N.
  $$
  Then $ - D^{\text m}$ is $\bQ$-linearly $f^{\text m}$-equivalent to the difference
  $N-T$, %:
%  $$
%  - D^{\text m} \sim_{f^{\text m},\q}   -\Delta^{\text m} +E^{\text
%    m} +F^{\text m} +\mu H^{\text m}_1,
%  $$
  hence it is $f^{\text m}$-nef. Since $D^{\text m}$ is $f$-exceptional, $f^{\text
    m}_*\bigl(D^{\text m}\bigr) =0$, so $D^{\text m}$ is effective by
  \cite[3.39]{KM98}.  
  % A divisor in $F$ appears in $\mu C+\Delta'$ with coefficient
  % $<1$ and in $(1+\nu)F$ with coefficient $(1+\nu)$ and so every divisor in $F$ has a
  % negative coefficient in $\Delta-E-(1+\nu)F-f^{-1}_*\Delta+\mu C$.  Therefore 
  Chosing $0<\mu\ll \nu\ll 1$ shows that both $F$ and $B$ \uj{disappear} on $X^{\text
    m}$, so every $f^{\text m}$-exceptional divisor has discrepancy $\leq -1$ and
  hence $(X^{\text m}, \Delta^{\text m})$ is indeed a minimal dlt model of
  $(X,\Delta)$ as defined in \eqref{demo:defs-and-not}.
\end{proof}

\begin{thm}\label{thm:higher-direct-images-split}
  Let $X$ be a smooth variety over $\bC$ and $D=\sum a_iD_i$ an effective, integral
  snc divisor.  Let $\sL$ be a line bundle on $X$ such that $\sL^m\simeq \sO_X(D)$
  for some $m>\max\{a_i\}$.  Let $f:X\to S$ be a projective morphism.  Then the
  sheaves $\myR^if_* \bigl(\omega_X\otimes \sL\bigr)$ are torsion-free for all $i$ and
  $$
  \myR f_* \bigl(\omega_X\otimes \sL\bigr)\simeq \sum_i \myR^if_* \bigl(\omega_X\otimes
  \sL\bigr)[-i].
  $$
\end{thm}

\begin{proof}
  If $D=0$, this is \cite[2.1]{Kollar86a} and \cite[3.1]{Kollar86b}.  The general
  case can be reduced to this as follows.  The isomorphism $\sL^m\simeq \o_X(D)$
  determines a degree $m$-cyclic cover $\pi:Y\to X$ with a $\mu_m$-action and
  $\omega_X\otimes \sL$ is a $\mu_m$-eigensubsheaf of $\pi_*\omega_Y$.  In general
  $Y$ has rational singularities. Let $h:Y'\to Y$ be a $\mu_m$-equivariant resolution
  and $g:Y'\to S$ the composition $f\circ \pi\circ h$. Then $\myR h_* \omega_{Y'}\simeq
  \omega_Y$, thus
  \begin{multline*}
    \myR f_*{\left(\pi_*\omega_Y\right)}{\simeq \myR f_* \myR \pi_*\omega_Y} \simeq
    \myR f_*\myR\pi_*\myR h_*\omega_{Y'}\simeq \\
    \simeq \myR g_* \omega_{Y'}\simeq \sum_i {\myR^ig_*\omega_{Y'}[-i]} {\simeq
      \uj{\sum_i} \myR^if_*\pi_* h_*\omega_{Y'}}[-i] \simeq \sum_i
    \myR^if_*\left(\pi_*\omega_{Y}\right)[-i]
  \end{multline*}
  by \cite[3.1]{Kollar86b} and because $\pi$ is finite.  As all of these isomorphisms
  are $\mu_m$-equivariant, taking $\mu_m$-eigensubsheaves on both sides, we obtain
  the desired statement. Notice that in particular we have proven that $\myR^if_*
  \bigl(\omega_X\otimes \sL\bigr)$ is a subsheaf of $\myR^ig_*\omega_{Y'}$ which is
  torsion-free by \cite[2.1]{Kollar86a}.
\end{proof}

\section{Splitting over the non-klt locus}
%\input split.tex

\begin{comment} Refs to \cite{Kollar86a,Kollar86b}.  The results there are for
  \sideremark{This still needs proof or reference} $\myR^if_*\omega_X$ but here they are
  used for $\myR^if_*L$ where $L\simq K_X+\Delta$ with $(X,\Delta)$ klt.  The reduction
  is easy since such an $L$ occurs as a direct summand of $\omega$ of a suitable
  cyclic cover.

  It would be a little easier for us if we could use the case when $(X,\Delta)$ dlt.
  The torsion freeness part is done by Fujino.  I tried, but failed, to prove the
  decomposition $\myR f_*L=\sum \myR^if_*L[-i]$. In the non-klt case these sheaves come from
  variations of mixed HS's, and they do not form a semi-simple category. However, the
  decomposition may be true.
\end{comment}

In the following theorem we show that the DB criterion \eqref{thm:db-criterion} holds
in an important situation.

%%\rsideremark{Proposed new version}
\begin{thm}\label{splitting.thm}
  Let $f:{{Y}}\to {{X}}$ be a proper morphism with connected fibers between normal
  varieties. Assume that $({{Y}},\Delta)$ is lc and $K_{{Y}}+\Delta\sim_{\q,f} 0$.
%  Let $X$ be a normal variety and $f:(Y,\Delta)\to X$ a \structure on $X$.  
%
  Set $W:=f(\nklt({{Y}},\Delta))$ and assume that $W\neq {{X}}$. \ujv {Let} $\pi:\wt
  {{Y}}\to {{Y}}$ \ujv{be a proper birational morphism and $F\leteq \wt
    f^{-1}(W)_{\ujv{\red}}$, where $\wt f\leteq f\circ \pi$}.  Then the natural map
  $$
  \varrho: \sI_W\simeq \wt f_*\sO_{{\wt Y}}(-F)\to \myR\wt f_*\sO_{{\wt Y}}(-F)
  $$
  has a left inverse. 
  % I.e., there exists a map $\varrho':\myR f_*\sO_{{\wt Y}}(-F)\to f_*\sO_{{\wt Y}}(-F)$
  % such that $\varrho'\circ\varrho=\id_{\sI_W}$.
\end{thm}

\begin{proof}
  First, observe that if $\tau: \what Y\to \wt Y$ is a log resolution of $(\wt Y,F)$
  that factors through $\pi$, then it is enough to prove the statement for
  $\sigma=\pi\circ\tau$ instead of $\pi$. Indeed let $\what F=\tau^{-1}F$, an snc
  divisor, and $\what f=f\circ\sigma$.  Suppose that the natural map
  $$
  \what\varrho:\sI_W\simeq \what f_*\sO_{{\what Y}}(-\what F)\to \myR\what
  f_*\sO_{{\what Y}}(-\what F)
  $$
  has a left inverse, $\what\delta: \myR\what f_*\sO_{{\what Y}}(-\what F)\to \what
  f_*\sO_{{\what Y}}(-\what F)$ such that $\what\delta\circ \what\varrho =
  \id_{\sI_W}$.  Then, as $\what f=\wt f\circ \tau$, one has that $\myR\what
  f_*\sO_{{\what Y}}(-\what F)\simeq \myR\wt f_*\myR\tau_*\sO_{{\what Y}}(-\what F)$ and
  applying the functor $\myR\wt f_*$ to the natural map $\ol\varrho:\sO_{\wt
    Y}(-F)\simeq \tau_*\sO_{\what Y}(-\what F) \to \myR\tau_*\sO_{\what Y}(-\what F)$
  shows that $\what\varrho=\myR\wt f_*(\ol\varrho)\circ \varrho$ :
  $$
  \xymatrix{%
    \what\varrho: \sI_W\simeq \wt f_*\sO_{{\wt Y}}(-F) \ar[rr]^-\varrho && \myR\wt
    f_*\sO_{{\wt Y}}(-F) \ar[rr]^{\myR\wt f_*(\ol\varrho)} && \myR\what f_*\sO_{{\what
        Y}}(-\what F).  }
  $$
  Therefore, $\delta=\what\delta\circ \myR\wt f_*(\ol\varrho)$ is a left inverse to
  $\varrho$ showing that it is indeed enough to prove the statement for $\sigma$.  In
  particular, we may replace $\pi$ with its \ujjj{composition} with any further blow
  up. We will use this observation throughout the proof.

  Next write
  $$
  \pi^*( K_{{Y}}+\Delta)\simq K_{\wt {{Y}}}+E+\wt\Delta-B,
  $$ 
  where $E$ is the sum of all (not necessarily exceptional) divisors with discrepancy
  $-1$,
  % $\wt \Delta=\pi^{-1}_*\Delta-\rup{\pi^{-1}_*\Delta}$ 
  $B$ is an effective exceptional integral divisor, and $\rdown{\wt \Delta}=0$.  We
  may assume that $\wt f^{-1}\wt f(E)$ is an snc divisor.  Since $B-E\geq -F$, we
  have natural maps
  $$
  \wt f_*\sO_{\wt {{Y}}}(-F)\to \myR\wt f_*\sO_{\wt {{Y}}}(-F)\to \myR\wt f_*\sO_{\wt
    {{Y}}}(B-E).
  $$
  Note that $B-E\sim_{\q,\wt f}K_{\wt {{Y}}}+\wt\Delta$, hence by
  \eqref{thm:higher-direct-images-split}
  $$
  \myR\wt f_*\sO_{\wt {{Y}}}(B-E)\cong \sum_i \myR^{i}\wt f_*\sO_{\wt {{Y}}}(B-E)[-i].
  $$
  Thus we get a morphism
  $$
  \wt f_*\sO_{\wt {{Y}}}(-F)\to \myR\wt f_*\sO_{\wt {{Y}}}(-F) \to \myR\wt f_*\sO_{\wt
    {{Y}}}(B-E)\to \wt f_*\sO_{\wt {{Y}}}(B-E).
  $$
  Note that $\pi_*\sO_{\wt {{Y}}}(B-E)=\sI_{\nklt({{Y}},\Delta)}$. Furthermore, for
  any $U\subseteq {{X}}$ open subset with preimage $U_{{Y}}:=f^{-1}(U)$, a global
  section of $\sO_{U_{{Y}}}$ vanishes along a fiber of $f$ if and only if it vanishes
  at one point of that fiber. Thus

  \hfill $ \wt f_*\sO_{\wt {{Y}}}(-F)=f_*\sI_{\nklt({{Y}},\Delta)}=\wt f_*\sO_{\wt
    {{Y}}}(B-E).  $\hfill
\end{proof}

\section{Log canonical centers}

We need the following higher dimensional version of a result of Shokurov
\cite[12.3.1]{K+92}, cf.\ \cite{MR1756108}.

\begin{prop} \label{2comp->plt.prop}%
  Let $f:{{Y}}\to {{X}}$ be a proper morphism with connected fibers between normal
  varieties. Assume that $({{Y}},\Delta)$ is lc and $K_{{Y}}+\Delta\sim_{\q,f} 0$.
  \ujj{For an arbitrary $x\in {{X}}$ let $U$ denote an \'etale local neighbourhood of
    $x\in X$. Then $f^{-1}(U)\cap \nklt({{Y}},\Delta)$} is either
  \begin{enumerate}
  \item connected, or
  \item has 2 connected components, both of which dominate \ujj{${{U}}$} and
    $({{Y}},\Delta)$ is plt near $f^{-1}(x)$.
  \end{enumerate}
\end{prop}

\begin{proof} 
  We may replace $(Y,\Delta)$ by a $\q$-factorial dlt model by \eqref{thm:dlt-models}
  \ujj{and assume that $X=U$.}
  %%% After restricting to $U$
  %%% an \'etale local neighborhood of $x\in X$   
  Then we may also assume that $\nklt({{Y}},\Delta)$ and $f^{-1}(x)\cap
  \nklt({{Y}},\Delta)$ have the same number of connected components.

  Write $\Delta=E+\Delta'$ where $E=\nklt({{Y}},\Delta)=\rdown{\Delta}$ and
  $({{Y}},\Delta')$ is klt. Let $E=\sum E_i$ be the decomposition to a sum of the
  connected components.  Pushing forward
  $$
  0\to \o_{{Y}}(-E)\to \o_{{Y}}\to \o_E\to 0
  $$
  we obtain
  $$
  0\to f_*\o_{{Y}}(-E)\to \o_{{X}}\to \sum_if_*\o_{E_i}\to \myR^1f_*\o_{{Y}}(-E).
  $$
  Note that $-E\sim_{\q,f}K_{{Y}}+\Delta'$, hence $\myR^1f_*\o_{{Y}}(-E)$ is torsion
  free by \eqref{thm:higher-direct-images-split} and applying \cite[2.54]{Fujinobook}
  (cf.\ \cite[3.2]{MR1993751}) to a resolution of the dlt pair $(Y,\Delta)$.

  Suppose $E_1$ does not dominate ${{X}}$. Then $f_*\o_{E_1}$ is a nonzero torsion
  sheaf, hence the induced map $f_*\o_{E_1}\to \myR^1f_*\o_{{Y}}(-E)$ must be zero.
  This implies that 
  $$
  f_*\sO_{E_1}\subseteq \im\bigl[\sO_X\to \sum_{i} f_*\sO_{E_i}\bigr].
  $$ %
  \ujvj{Since we are working locally near $x\in X$, we may assume that
    $(f_*\sO_{E_1})_x\neq 0$.} %
  Observe that the natural projection map $\sum_{i} f_*\sO_{E_i}\to f_*\sO_{E_1}$
  gives a splitting of \uj{the above} embedding.  Further observe, that
  $\im\bigl[\sO_X\to \sum_{i} f_*\sO_{E_i}\bigr]$ has only one generator near $x$.
  \uj{This implies that we must have that} $f_*\sO_{E_1}= \im\bigl[\sO_X\to \sum_{i}
  f_*\sO_{E_i}\bigr]$ \uj{ locally near $x$}.  In particular, there is at most one
  $E_i$ that does not dominate ${{X}}$. Furthermore, if $E_j$ does dominate ${{X}}$
  then $\o_{{X}}\to f_*\o_{E_j}$ is nonzero.  This again would contradict
  $f_*\sO_{E_1}= \im\bigl[\sO_X\to \sum_{i} f_*\sO_{E_i}\bigr]$.  Therefore, if $E$
  has more than one component, then they all dominate ${{X}}$.

  Until now the statement and the proof could have been done birationally, but for
  the rest we use the MMP repeatedly.  Note that the proof is a bit messier than
  \cite[12.3.1]{K+92} since we do not have the full termination of MMP.
 
  First we run the $({{Y}},(1-\varepsilon) E+\Delta')$-MMP cf.~\cite[1.3.2]{BCHMc06}.
  Every step is numerically $K_{{Y}}+\Delta$-trivial, hence, by the usual
  connectedness (cf.\ \cite[5.48]{KM98}) the $E_i$ stay disjoint.  At some point, we
  must encounter a Fano-contraction $\gamma:({{Y}}^*,(1-\varepsilon)E^*+\Delta^*)\to
  S$ where $E^*$ is ample on the general fiber.  As we established above, every
  connected component of $E^*$ dominates $S$. We may assume that $E^*$ is
  disconnected as otherwise we are done.  

  Since the relative Picard number \ujjj{of $Y$} is \ujjj{$1$}, every connected
  component of $E^*$ is relatively ample. As $E^*$ is disconnected, all fibers are
  $1$-dimensional. Since $\gamma$ is a Fano-contraction, the generic fiber is $\bP^1$
  and so $E^*$ can have at most, and hence exactly, two connected components, $E^*_1$
  and $E^*_2$. Since the fibration is numerically $K_{{Y^*}}+\Delta^*$-trivial, it
  follows, that the intersection product of either $E^*_i$ with any fiber is $1$. In
  other words, the $E^*_i$ are sections of $\gamma$. Since they are also relatively
  ample, it follows that every fiber is irreducible and so outside a codimension $2$
  set on the base, $\gamma:Y^*\to S$ is a $\bP^1$-bundle with two disjoint sections.
  It also follows that $\Delta^*$ does not intersect the general fiber, hence
  $\Delta^*=\gamma^*\Delta_S$ for some $\Delta_S\subset S$ and then since the $E^*_i$
  are sections we have that $(E^*_i, \Delta^*\resto{E^*_i})\simeq (S, \Delta_S)$.

  We need to prove that $(Y^*, E^*_1+E^*_2+\Delta^*)$ is plt and for that it is
  enough to show that $(E^*_i, \Delta^*\resto{E^*_i})$ is klt for $i=1,2$. By the
  % 
  %% \sideremark{I added $\Gamma$ here, because I believe it is possible that we need
  %%   that. We need that $\Gamma$ is effective, but I think that's OK. That's where
  %%   we use that the sections are disjoint.}
  %
  above observation, all we need to prove then is that $(S, \Delta_S)$ is klt.  Since
  $\gamma$ is a $\bP^1$-bundle (in codimension $1$) with $2$ disjoint sections, we
  have that $K_{Y^*}+E^*_1+E^*_2\sim \gamma^*K_S$ and then that
  $K_{Y^*}+E^*_1+E^*_2+\Delta^*\sim_{\bQ} \gamma^*(K_S+\Delta_S)$.
  Now we may apply \cite[20.3.3]{K+92} to a general section of $Y^*$ mapping to $S$
  to get that $(S, \Delta_S)$ is klt.
\begin{comment}
  We still need to prove that $({{Y}},\Delta)$ is plt if $E$ has 2 connected
  components.  Since the discrepancies are unchanged as we go from $({{Y}},\Delta)$
  to $({{Y^*}}, E^*+\Delta^*)$, it is enough to prove that $({{Y^*}}, E^*+\Delta^*)$
  is plt.  The key property we use is that the generic fiber of ${{Y^*}}\to S$ is
  $\bP^1$ and and intersects $E^*$ in at most $2$ points.

  If $({{Y^*}}, E^*+\Delta^*)$ is not plt, then we can pass to another dlt model
  ${{Y}}^*_1\to {{Y}}^*$ that introduces new exceptional divisors with discrepancy
  $-1$.  Thus on ${{Y}}^*_1$ we have $E_1^*, E_2^*$ and $F_1,\dots, F_r$ where the
  $E_i^*$ dominate $S$ and $F_1,\dots, F_r$ do not dominate $S$.  Note also that
  ${{Y^*_1}}/S$ has Picard number $r+1$ and so $E_1^*, E_2^*$ and $F_1,\dots,
  F_{r-1}$ generate the relative Picard group.  Now run the
  $$
  ({{Y}}^*_1, E_1^*+E_2^*+F_1+\cdots + F_{r-1}+\Delta^*_1)-\mbox{MMP over $S$.}
  $$
  Every step of the MMP intersects the reduced boundary.  As proved in
  \cite[Lemma~5.1]{BCHMc06} this implies termination of flips.  Thus we stop when the
  birational transform of $-F_r$ becomes relatively nef.  Since $F_r$ does not
  dominate $S$, this can only happen if this birational transform becomes a union of
  fibers.  This, however, contradicts our previous observation that the 2
  disconnected components stay disconnected.
\end{comment}
\end{proof}

\noindent
We are now ready to prove our main connectivity theorem.

\begin{comment}
  \begin{thm}\label{lc.cent.thm}
    Let $f:{{Y}}\to {{X}}$ be a proper morphism with connected fibers between normal
    varieties. Assume that $({{Y}},\Delta)$ is lc and $K_{{Y}}+\Delta\sim_{\q,f} 0$.
    Let $Z_1, Z_2\subset Y$ be lc centers of $(Y,\Delta)$.
    % 
    Then every irreducible component of $Z_1\cap f^{-1}\bigl(f(Z_2)\bigr)$ is also an
    lc center of $(Y,\Delta)$.
  \end{thm}
\end{comment}

\begin{newnum}{\it Proof of \eqref{lc.cent.thm}}
  We may assume that $f$ is surjective and replace $(Y,\Delta)$ by a $\q$-factorial
  dlt model by \eqref{thm:dlt-models}.  If $Z_1=Y$ then $Z_2\subseteq Z_1$, and if
  $f(Z_2)=X$ then $Z_1\subseteq Z_1$ satisfy the requirement, hence we may assume
  that $(Y,\Delta)$ is dlt, $Z_1, Z_2\subset \rdown{\Delta}$ are divisors, and $Z_2$
  is disjoint from the generic fiber of $f$.  Then, by localizing at a generic point
  of $f(Z_1)\cap f(Z_2)$ we reduce to the case when $x:=f(Z_1)\cap f(Z_2)$ is a
  closed point.

  By working in a suitable \'etale neighborhood of $x$, we may also assume that
  $Z_1\cap f^{-1}(x)$ is geometrically connected.  Thus it is sufficient to prove
  that $Z_1\cap f^{-1}(x)$ contains an lc center.

  Since we are now assuming that $Z_2$ does not dominate $X$, it follows from
  (\ref{2comp->plt.prop}) that $f^{-1}(x)\cap \rdown{\Delta}$ is connected, and hence
  there are irreducible divisors
  $$
  V_1:=Z_2, V_2,\dots, V_{m-1}, V_m:=Z_1\qtq{with} V_i\subset \rdown{\Delta}
  $$
  such that $f^{-1}(x)\cap V_i\cap V_{i+1}\neq \emptyset$ for $i=1,\dots, m-1$.  By
  working in the \'etale topology on $X$, we may also assume that each $f^{-1}(x)\cap
  V_i$ is connected.

  Next, we prove by induction on $i$ that 
  \begin{equation}
    \label{eq:3}
    W_i:= V_i\cap \bigcap_{j<i}f^{-1}\bigl(f(V_j)\bigr) %\eqno{(\ref{lc.cent.thm}.1)}
  \end{equation}
  contains an lc center of $(Y,\Delta)$.

  For $i=1$ the statement of (\ref{eq:3}) follows from the fact that $V_1=Z_2$ is an
  lc center of $(Y,\Delta)$.  Next we go from $i$ to $i+1$.  Consider $\bigl(V_i,
  \diff_{V_i}(\Delta-V_i)\bigr)$.  Note that every irreducible component of $V_i\cap
  V_{i+1}$ is an lc center of $\bigl(V_i, \diff_{V_i}(\Delta-V_i)\bigr)$ and by
  induction and adjunction $W_i$ contains an lc center of $\bigl(V_i,
  \diff_{V_i}(\Delta-V_i)\bigr)$.  Thus, by induction on the dimension, replacing $Y$
  by $V_i$, $Z_1$ by $V_i\cap V_{i+1}$, and $Z_2$ by the lc center contained in
  $W_i$, we conclude that $f^{-1}\bigl(f(W_i)\bigr)\cap V_i\cap V_{i+1}$ contains an
  lc center $U_i$ of $ \bigl(V_i, \diff_{V_i}(\Delta-V_i)\bigr)$.  By inversion of
  adjunction, $U_i$ is also an lc center of $(Y,\Delta)$ and it is contained in
  $W_{i+1}$.
  % $$
  % \bigl(V_i\cap V_{i+1}, \diff_{V_i\cap V_{i+1}}(\Delta-V_i-V_{i+1})\bigr).
  % $$
  % and using inversion of adjunction, $U_i$ is an lc center of
  % $$
  % \bigl(V_{i+1}, \diff_{V_{i+1}}(\Delta-V_{i+1})\bigr).
  % $$
  % By construction, $U_i\subset W_{i+1}$.
 
  At the end we obtain that
  \begin{equation}
    \label{eq:7}
      W_m=Z_1\cap f^{-1}\bigl(f(Z_2)\bigr)\cap f^{-1}\bigl(f(V_2)\bigr)\cap\cdots\cap
      f^{-1}\bigl(f(V_{m-1})\bigr) %\eqno{(\ref{lc.cent.thm}.2)}
  \end{equation}
  contains an lc center of $(Y,\Delta)$. Observe that $W_m$ contains $Z_1\cap
  f^{-1}(x)$ and is contained in $Z_1\cap f^{-1}\bigl(f(Z_2)\bigr)$.  These two are
  the same, hence we are done.
\end{newnum}

\begin{subrem}
  The statement of \eqref{lc.cent.thm} is stronger than that has been previously
  known \cite[1.5]{MR1457742}, \cite[4.8]{MR1993751}, \cite[3.45]{Fujinobook}.  The
  usual claim in a similar situation has been that every irreducible component of
  $f(Z_1)\cap f(Z_2)$ is dominated by an lc center, whose precise location was not
  known.

  It would also be interesting to find a proof of \eqref{lc.cent.thm} without using
  the MMP.
\end{subrem}

\begin{comment}
  One hopes that a similar statement is true without assuming $K_Y+\Delta\sim_f 0$,
  only some birational variant. One might try to prove that every irreducible
  component of $f(Z_1)\cap f(Z_2)$ is dominated by an lc center contained in $Z_1$.

  One apparent obstacle is the use of the MMP in the above proof. If one could avoid
  that then the proof would work in the qlc setting.
\end{comment}

\begin{defini}\label{def:KK-morphism}
  Let $X$ be a normal scheme. A \emph{minimal quasi log canonical structure} or
  simply a \emph{\structure} on $X$ is a proper surjective morphism $f: (Y,\Delta)\to
  X$ where
  \begin{enumerate-p}
  \item  $(Y,\Delta)$ is a log canonical pair,
  \item  $\Delta$ is effective, \label{item:effective}
  \item  $\sO_X\simeq f_*\sO_Y$, and
  \item  $K_Y+\Delta\sim_{f,\bQ}0$. \label{item:minimal}
  \end{enumerate-p}
%  A \structure is called \emph{dlt} if $(Y,\Delta)$ is dlt.
\end{defini}

\begin{subrem}
  This definition is similar to Ambro's %and Fujino's
  definition of a quasi-log variety \cite[4.1]{MR1993751}, \cite[3.29]{Fujinobook}.
  The main difference here, underscored by the word ``minimal'' in the definition, is
  the additional assumption
  %s (\ref{def:KK-morphism}.\ref{item:effective}) and
  (\ref{def:KK-morphism}.\ref{item:minimal}).

  One should also note that what Fujino calls a quasi-log variety is essentially $X$
  together with a qlc stratification which we define next.
%  cf.\ \eqref{def:qlc-strata}.
%% , not present in  Fujino's definition, requiring that
%% $K_Y+\Delta\sim_{f,\bQ}0$. The name is  motivated by the fact that
%% (\ref{def:KK-morphism}.\ref{item:minimal}) holds  typically if $X$ is log
%% canonical (with an appropriate divisor) and $(Y,\Delta)$ is  a minimal model of a
%% log resolution of $X$.  
\end{subrem}

\begin{defini}\label{def:qlc-strata}
  Let $X$ be a normal scheme and assume that it admits a \structure $f: (Y,\Delta)\to
  X$.  We define the \emph{qlc stratification} of $X$ \emph{with respect to $f$} or
  simply the \emph{\stratification} the following way: Let $\mcH_Y$ denote the set
  containing all the lc centers of $(Y,\Delta)$, including the components of $\Delta$
  and $Y$ itself.  For each $Z\in \mcH_Y$ let
  $$
  W_Z\leteq f(Z)\setminus \bigcup_{%\footnotesize
    \begin{matrix}
      {Z'\in\mcH_Y} \\ f(Z)\not\subseteq f(Z')
    \end{matrix}
  } f(Z').
  $$
  Further let 
  $$
  \HX=\{ W_Z | Z\in \mcH_Y \}.
  $$
  Notice that it is possible that $W_Z=W_{Z'}$ for some $Z\neq Z'$, but in $\HX$
  they are only counted once. 
  Then 
  $$
  X=\coprod_{W\in \HX%\cup {\emptyset}
  } W
  $$
  will be called the \emph{qlc stratification} of $X$ \emph{with respect to $f$} and
  the strata the \emph{\strata.} \uj{Note that by construction \ujjj{each \stratum}
    is reduced.}
\end{defini}

\begin{defini}
  Let $X_i$ be varieties that admit \structure\nospace s, $f_i: (Y_i,\Delta_i)\to
  X_i$ and let $W_i=\cup_{j=1}^{r_i}W_{i,j}$ be unions of some \strata on $X_i$ for
  $i=1,2$. Assume that there exists a morphism $\alpha: W_1\to W_2$. Then we will say
  that $\alpha$ is a \emph{qlc stratified morphism} \ujjj{if for every \stratum $W_{2,j}$,
  its preimage $\alpha^{-1}W_{2,j}$ is equal to a disjoint union of \strata
  $\cup_{\alpha}W_{1,j_\alpha}$ for an appropriate set of $\alpha$'s.}
\end{defini}

Using our new terminology we have the following important consequence of
\eqref{lc.cent.thm}.

\begin{cor}\label{cor:closure-of-union-is-union}
  Let $X$ be a normal variety with a \structure, $f: (Y,\Delta)\to X$.  Then the
  closure of any union of some \strata is also a union of some \strata.
\end{cor}
\begin{proof}
  It is enough to prove this for the closure of a single \stratum. By definition, the
  difference between the closure and the \stratum is a union of intersections of that
  single stratum with the images of lc centers. By \eqref{lc.cent.thm} this is
  covered by a union of \strata.
\end{proof}
 
In \eqref{rem:db-is-sn} we observed that DB singularities are seminormal, so it
follows from Theorem~\ref{thm:main-proof} that the closure of any union of \strata is
seminormal.  On the other hand it also follows from the somewhat simpler
\eqref{cor:closure-of-union-is-union} and similar results from \cite{Fujinobook}.

\begin{prop}[\protect{\cite{MR1993751}, \cite[\S
    3]{Fujinobook}}]\label{prop:seminormal} 
  Let $X$ be a normal variety that admits a \structure, $f: (Y,\Delta)\to X$.  Then
  each \stratum is normal and the closure of any union of \strata is seminormal.
\end{prop}

\begin{proof}
  Let $T$ be the closure of a union of some \strata.  Then by
  Corollary~\ref{cor:closure-of-union-is-union} and \cite[3.39(i)]{Fujinobook} (cf.\
  \cite[4.4]{MR1993751}) the qlc centers of $T$ are exactly %the closures of
  the \strata (of $X$) that lie inside $T$.  It follows by \cite[3.33]{Fujinobook}
  that $T$ is seminormal and by \cite[3.44]{Fujinobook} (cf.\ \cite[4.7]{MR1993751})
  that each \stratum is normal.
\end{proof}

\begin{cor}\label{cor:conductor}
  Let $X$ be a normal variety with a \structure, $f: (Y,\Delta)\to X$.  Then the
  support of the conductor subscheme of the closure of any union of \strata is
  {contained in} a smaller dimensional union of \strata.
\end{cor}

\begin{proof}
  As individual \strata are normal, it follows that the conductor subscheme is
  contained in the part of the closure that was subtracted in \eqref{def:qlc-strata}.
  By \eqref{lc.cent.thm} this is a union of \strata and as it does not contain any
  (maximal) component of the original union, the dimension of each contributing
  strata has to be strictly smaller.
\end{proof}

\section{Log canonical singularities are Du~Bois}
\label{sec:lc-is-db}

\begin{lem}\label{lem:key-lemma}
  Let $X$ be a normal variety and $f:(Y,\Delta)\to X$ a \structure on $X$. Let $W\in
  \HX$ be a qlc stratum of $X$ and $\ol W$ its closure in
  $X$. % Assume that $\ol W\neq X$.
  Then there exist a normal variety $\what W$ with a \structure $g:(Z,\Sigma)\to
  \what W$ such that $g(\nklt(Z,\Sigma))\neq \what W$ and a finite surjective qlc
  stratified morphism $\what W\to \ol W$.
\end{lem}

\begin{proof}
  \ujj{%
  We will repeat the following procedure until all the desired conditions are
  satisfied.
  
  {\sc Iteration:}} Note that we may replace $(Y,\Delta)$ by a $\q$-factorial dlt
model by \eqref{thm:dlt-models}.  Recall that in that case $\ol W$ is the union of
some \strata by \eqref{cor:closure-of-union-is-union}. 
  If $\ol W=X$ and $\ujj{f(\nklt(Y,\Delta))}\neq X=\ol W$ then choosing
  $(Z,\Sigma)=(Y,\Delta)$, $g=f$, and $\what W=X$ the desired conditions are
  satisfied.  Otherwise, there exists an irreducible component $E\subseteq
  \rdown\Delta$ such that $\ol W\subseteq {f(E)}$.  Consider the Stein factorization
  of $f\resto E$ :
  $$
  \xymatrix{f\resto{E}:E\ar[r]^-{f_E} &{G}\ar[r]^-\sigma & {f(E)}}.
  $$ 
  Observe that then $f_E: (E,\diff_E\Delta)\to {G}$ is a \structure, ${G}$ is normal,
  and $\sigma$ is finite. Let $W_1=\sigma^{-1}(W)$ denote the preimage of $W$, and
  $\ol W_1$ its closure in $G$.  By \eqref{lc.cent.thm} the $f_E$-qlc stratification
  of ${G}$ is just the preimage of the restriction of the $f$-qlc stratification of
  $X$ to $f(E)$, so the induced morphism $\ol W_1\to \ol W$ is a qlc stratified
  morphism and as long as $\ol W\neq {f(E)}$ or $f(\nklt(E,\diff_E\Delta))=f(E)$ we
  \ujj{%
    may go back to the beginning and repeat our procedure } with $X$ replaced with
  ${G}$ and $W$ replaced with $\sigma^{-1}(W)$ without changing the induced qlc
  structure on $W$. By noetherian induction this process must end \ujjj{and then} we
  will have $\ol W= {f(E)}$ and $f(\nklt(E,\diff_E\Delta))\neq f(E)$.  Then $f_E:
  (E,\diff_E\Delta)\to {G}$ and $\sigma:\what W:={G}\to\ol W$ satisfy the desired
  conditions.
\end{proof}

Theorem~\ref{thm:main} is implied by the following.

\begin{thm}\label{thm:main-proof}
  If $X$ admits a \structure, $f: (Y,\Delta)\to X$, then the closure of any union of 
  \strata is DB.
  %(and hence it is seminormal cf.~\eqref{rem:db-is-sn},  \eqref{prop:seminormal}). 
\end{thm}

\begin{proof}
  Let $T\subseteq X$ be a union of \strata.  By \eqref{cor:closure-of-union-is-union}
  we know that $\ol T$, the closure of $T$ in $X$, is also a union of \strata, so by
  replacing $T$ with $\ol T$ we may assume that $T$ is closed. Let $\wt T$ denote the
  normalization of $T$. We have that $\dsize T=\bigcup_{W\in\mcJ} W$ for some
  $\mcJ\subseteq \HX$, so $T$ is seminormal by \eqref{prop:seminormal}.  For
  $W\in\mcJ$, we will denote the closure of $W$ in $X$ by $\ol W$. Note that by
  definition $\ol W$ is contained in $T$. In order to prove that $T$ is DB, we will
  apply a double induction the following way:
  \begin{itemize}
  \item {\it induction on $\dim X$:} Assume that the statement holds if $X$ is
    replaced with a smaller dimensional variety admitting a \structure.
  \item {\it induction on $\dim T$:} Assume that the statement holds if $X$ is fixed
    and $T$ is replaced with a smaller dimensional subvariety of $X$ which is also a
    union of \strata.
  \end{itemize}

  First assume that $X\neq T$. Then $\ol W$ must also be a proper subvariety of $X$
  for any $W\in\mcJ$. Then by \eqref{lem:key-lemma} for each $W\in\mcJ$ there exists
  a normal variety $\what W$ with a \structure and a finite surjective qlc stratified
  morphism $\sigma: \what W\to \ol W$.  By induction on $\dim X$ we obtain that
  $\what W$ is DB.  Then by \eqref{cor:finite-cover-to-normalization} it follows that
  the normalization of $\ol W$ is DB as well. Note that $\what W$ is normal, but may
  not be the normalization of $\ol W$, however $\sigma$ factors through the
  normalization morphism.

  Let $\mcJ'\subseteq\mcJ$ be a subset such that $\dsize T=\bigcup_{W\in\mcJ'} \ol W$
  and $\ol W\not\subseteq \ol {W'}$ for any $W, W'\in\mcJ'$. Then let $\dsize \what
  T\leteq\coprod_{W\in\mcJ'} \what W$ and $\what\tau: \what T\to T$ the natural
  morphism.  Observe that as the $\what W$ have DB singularities, so does $\what T$
  and then by \eqref{cor:finite-cover-to-normalization} it follows that for the
  normalization of $T$, $\tau: \wt T\to T$, $\wt T$ is DB as well.
  Next let $Z\subset T$ be the conductor subscheme of $T$ and $\wt Z$ its preimage in
  $\wt T$.  Then since $T$ is seminormal, both $Z$ and $\wt Z$ are reduced and
  \begin{equation}
    \label{eq:8}
    \sI_{Z\subseteq T}= \tau_*\sI_{\wt Z\subseteq \wt T}.
  \end{equation}

  \uj{%
  \begin{subclaim}\label{claim:overset-of-conductor}
    %%% Let $T$ be a seminormal scheme, $\tau:\wt T\to T$ its normalization,
    %%% $Z\subseteq T$ a reduced subscheme and $\wt Z$ the preimage of $Z$ in $\wt
    %%% T$. Assume that $\sI_{Z\subseteq T}= \tau_*\sI_{\wt Z\subseteq \wt T}$.
    Let $\Gamma\subseteq T$ be a reduced % and irreducible
    subscheme that contains the %general point of an irreducible component of the
    conductor $Z$ and let $\wt \Gamma$ be its preimage in $\wt T$. Then
    $\sI_{\Gamma\subseteq T}\subseteq \sO_T\subseteq \tau_*\sO_{\wt T}$ is also a
    $\tau_*\sO_{\wt T}$ ideal, i.e., $\sI_{\Gamma\subseteq T}=\sI_{\Gamma\subseteq
      T}\cdot \tau_*\sO_{\wt T}$. In particular,
    \begin{equation}
      \label{eq:14}
      \sI_{\Gamma\subseteq T}= \tau_*\sI_{\wt \Gamma\subseteq \wt T}.
    \end{equation}
  \end{subclaim}

  \begin{proof}
    If $\sJ=\sI_{\Gamma\subseteq T}$ is a $\tau_*\sO_{\wt T}$ ideal, then
    (\ref{eq:14}) follows, so it is enough to prove the first statement.  \ujjj{%
      Let $\sI=\sI_{Z\subseteq T}$. Clearly, $\sJ\cdot\tau_*\sO_{\wt T}\subseteq
      \sI\cdot \tau_*\sO_{\wt T} =\sI\subseteq \sO_T$.  Then $\sJ\cdot \tau_*\sO_{\wt
        T}\subseteq \sO_T\cap \sqrt{\sJ\cdot\tau_*\sO_{\wt T}}$, which is equal to
      $\sqrt{\sJ}$ by \cite[5.14]{MR0242802}. In turn, $\sqrt \sJ=\sJ$ by assumption,
      so we have that $\sJ\cdot\tau_*\sO_{\wt T}\subseteq \sJ$.  }
  \end{proof}
}
  
  \noindent 
  \uj{%
    % Now we continue the proof of \eqref{thm:main-proof}.  
    By \eqref{cor:conductor} $Z$ is contained in a union of \strata whose dimension
    is smaller then $\dim T$.  Replace $Z$ by this union and $\wt Z$ by its reduced
    preimage on $\wt T$. Then $Z$ is DB by induction on $\dim T$.  In the sequel we
    are only going to use one property of $Z$ that followed from being the conductor,
    namely the equality in (\ref{eq:8}). However, by
    \eqref{claim:overset-of-conductor} this remains true for the new choice of $Z$. %
  }%
  Next % consider $\wt Z$ and
  let $\what Z=(\what\tau^{-1}Z)_{\red}\subset \what T$ be the reduced preimage of
  $Z$ (as well as of $\wt Z$) in $\what T$. The following diagram shows the
  connections between the various objects we have defined so far:

  $$
  \xymatrix{
    \save[]+<-3em,0em>*\txt<8pc>{%
      \tiny normal, admits \structure} \restore
    \ar@{.>}@/_{.75pc}/[]+<.5em,-.75em>;[r]+<-1.5em,-.75em> 
    &     \hskip.5em \coprod \what W  = \what T \ar@<2ex>[dd]_{\hat \tau} \ar[dr]
    \supset     \what Z
    & \save[]+<1.95em,0em>*\txt<6pc>{%
      \tiny preimage of \uj{$Z$}} \restore \ar@{.>}[l]
    \\ 
      \save[]+<-.75em,0em>*\txt<6pc>{%
      \tiny finite} \restore \ar@{.>}[r]
    && \wt T \ar[dl]_{\tau} &  \hskip -3em \supset \wt Z  
    & \save[]+<1.8em,0em>*\txt<6pc>{%
      \tiny preimage of \uj{$Z$}} \restore \ar@{.>}[]+<-.5em,-0.1em>;[l]+<-.5em,-0.1em>\\
    X \hskip-3em & \hskip-.6em \supseteq \bigcup \ol W = T \supset Z \ &
    \save[]+<6em,0em>*\txt<12pc>{%
      \tiny \uj{the smaller dimensional union of \strata containing the conductor of $T$}}
    \restore \ar@{.>}[l] \save[]+<3.75em,1.7em>*\txt<6pc>{%
      \tiny normalization of $T$} \ar@{.>}[]+<1em,1.75em>;[]+<-2.8em,1.75em> \restore
  }
  $$

%%%  We have seen above that each $\what W$ admits a \structure compatible with the
%%%  \ujj{part of the} \structure \ujj{of $X$ that lies in} $T$ and then by
%%%  \eqref{cor:conductor} again 
  \ujj{%
    As we replaced the conductor with a union of \strata it was contained in and as}
  each $\what W$ admits a \structure compatible with the {part of the} \structure
  \ujj{of $X$ that lies in $T$, it follows that} $\what Z$ is also a union of qlc
  strata on $\what T$ and the morphism $\what Z\to \wt Z$ is a qlc stratified
  morphism. Then since $\dim \what T< \dim X$, by replacing $X$ with $\what T$ shows
  that $\what Z$ is DB by induction on $\dim X$. In turn this implies that $\wt Z$ is
  DB by \eqref{cor:finite-cover-is-ok}.

  Therefore, by now we have proved that $\wt T$, $Z$, and $\wt Z$ all have DB
  singularities, so by using \uj{\eqref{claim:overset-of-conductor}} and
  \eqref{thm:db-criterion} we conclude that $T$ is DB as well.

  Now assume that $X=T$ and hence $X=T=\wt T$. Let $f:(Y,\Delta)\to X$ be a
  \structure and $W=f(\nklt(Y,\Delta))$. By \eqref{lem:key-lemma} we may assume that
  $W\neq X$ by replacing $X$ by a finite cover. Note that by
  \eqref{cor:finite-cover-is-ok} it is enough to prove that this finite cover is DB.

  Then let $\pi:\wt {{Y}}\to {{Y}}$ be a log resolution and $F\leteq
  (f\circ\pi)^{-1}(W)$, an snc divisor.  By \eqref{splitting.thm} the natural map $
  \varrho: \sI_W=f_*\o_{{\wt Y}}(-F)\to \myR f_*\o_{{\wt Y}}(-F)$ has a left inverse.
  Finally, then \eqref{thm:db-criterion} implies that $T=X$ is DB.
\end{proof}

\begin{defini}
  \label{def:db-morphism}
  Let $\phi:X\to B$ be a flat morphism.
  %  and assume that $B$ is connected, irreducible and reduced.
  We say that $\phi$ is a \emph{DB family} if $X_b$ is DB for all $b\in B$.
\end{defini}

\begin{defini}
  \label{def:flc-morphism}
  Let $\phi:X\to B$ be a flat morphism.
  %  and assume that $B$ is connected, irreducible and reduced.
  We say that $\phi$ is a \emph{family with potentially lc fibers} if for all closed
  points $b\in B$ there exists an effective $\bQ$-divisor $D_b\subset X_b$ such that
  $(X_b, D_b)$ is log canonical.
\end{defini}

\begin{defini}[\protect{\cite[7.1]{KM98}}]
  \label{def:lc-morphism}
  Let $X$ be a normal variety, $D\subset X$ an effective $\bQ$-divisor such that
  $K_X+D$ is $\bQ$-Cartier, and $\phi:X\to B$ a non-constant morphism to a smooth
  curve $B$. We say that $\phi$ is a \emph{log canonical morphism} or an \emph{lc
    morphism} if $(X, D+X_b)$ is lc for all closed points $b\in B$.
\end{defini}

\begin{rem}\label{rem:lc-is-db}
  Notice that for a family with potentially lc fibers it is not required that the
  divisors $D_b$ also form a family over $B$. On the other hand, if $\phi:X\to B$ is
  a family with potentially lc fibers, $B$ is a smooth curve and there exists an
  effective $\bQ$-divisor such that $K_X+D$ is $\bQ$-Cartier and $D\resto {X_b}=D_b$
  then $\phi$ is an lc morphism by inversion of adjunction \cite{MR2264806}.
%
%  Note however, that the reverse statement is not true.  For instance, it is not
%  necessarily true that the fibers of an lc morphism are normal, in particular they
%  are not necessarily lc.

  Further observe that if $\phi:(X,D)\to B$ is an lc morphism, then for any $b\in B$,
  choosing $(Y,\Delta)= (X,D+X_b)$ and $f:(Y,\Delta)\to X$ the identity of $X$ gives
  an \stratification of $X$ such that $X_b$ is a union of \strata.  In particular, it
  follows by \eqref{thm:main-proof} that $X_b$ is DB.  Note that if $X_b$ is
  reducible, then \eqref{thm:main} would not suffice here.
\end{rem}

\begin{cor}
  Let $\phi:X\to B$ be either a family with potentially lc fibers or an lc morphism.
  Then $\phi$ is a DB family.
\end{cor}

\begin{proof}
  Follows directly from \eqref{thm:main-proof}.
\end{proof}

\section{Invariance of cohomology for DB morphisms}

\newcommand{\kdot}{\bdot}
%{{{\,\begin{picture}(1,1)(-1,-2)\circle*{2}\end{picture}\ }}}

\newcommand\omegai{h^{-i}(\clx{\omega}_{\phi})}

\newcommand\omegaib{h^{-i}(\clx{\omega}_{X_b})}
\newcommand\omegait{h^{-i}(\clx{\omega}_{\phi_T})}
\newcommand\varrhoib{\varrho^{-i}_b}
\newcommand\varrhoit{\varrho^{-i}_T}

\newcommand\omegaj{h^{-j}(\clx{\omega}_{\phi})}
\newcommand\omegajb{h^{-j}(\clx{\omega}_{X_b})}
\newcommand\omegajt{h^{-j}(\clx{\omega}_{\phi_T})}
\newcommand\varrhojb{\varrho^{-j}_b}
\newcommand\varrhojt{\varrho^{-j}_T}

\noindent
The following notation will be used throughout this section.

\begin{demo}{Notation}\label{not:coh-sheaves-of-omega}
  Let $\pi: \bP^N_B\to B$ be a projective $N$-space over $B$, $\iota:X\into \bP^N_B$
  a closed embedding, and $\phi\leteq\pi\circ\iota$. Further let $\sO_\pi(1)$ be a
  relatively ample line bundle \ujj{on $\bP^N_B$}, % , and $\sL=\sO_\pi(1)\resto X$.
  denote by $\omega_{\phi}^\bdot$ the relative dualizing complex \ujjj{$\phi^!\sO_B$}
  and by $\omegai$ its $-i^\text{th}$ cohomology sheaf.  We will also use the
  notation $\omega_\phi\leteq h^{-n}(\clx{\omega}_\phi)$ where $n=\dim (X/B)$.
  Naturally these definitions automatically apply for $\pi$ in place of $\phi$ by
  choosing $\iota=\id_{\bP^N_B}$.
  %%Note also that $\supp \omegai=\supp(\omegai\otimes \sL^q)$ for all $i$ and $q$.

  %%% Finally, let $b\in B$ and let
  %%%   $$
  %%%   \varrhoib:\phi_*(\omegai)\otimes\kappa(b)\to\omegaib
  %%%   $$
  %%%   denote the natural map induced by restriction to the fiber $X_b$.
\end{demo}

\begin{lem}\label{lem:D-duality}
  % Assume that $\phi$ is projective.
  Let $b\in B$.
  Then 
  $$
  \omegai\simeq \sExt^{N-i}_{\bP^N_B}(\sO_X,\omega_\pi)\quad\text{ and } \quad
  \omegaib\simeq \sExt^{N-i}_{\bP^N_b}(\sO_{X_b},\omega_{\bP^N_b}).  
  $$
  In particular, $\omegai=0$ and $\omegaib=0$ if $i<0$ or $i>N$.
\end{lem}

\begin{proof}
  By Grothendieck duality (\cite[VII.3.3]{MR0222093}, cf.\
  \cite[III.7.5]{Hartshorne77}),
  \begin{equation*}
    % \label{eq:10}
    \omegai\simeq  % h^{-i}(\clx{\omega}_\phi) \simeq
    h^{-i}(\myR\sHom_{\bP^N_B}(\sO_X,\omega_\pi^\bdot)) \simeq 
    % \\ \simeq
    h^{-i}(\myR\sHom_{\bP^N_B}(\sO_X,\omega_\pi)[N]) \simeq
    \sExt^{N-i}_{\bP^N_B}(\sO_X,\omega_\pi).
  \end{equation*}
  The same argument obviously implies the equivalent statement for $\omegaib$.

  Furthermore, clearly $\sExt^{j}_{\bP^N_B}(\sO_X,\omega_\pi)=0$ and
  $\sExt^{j}_{\bP^N_b}(\sO_{X_b},\omega_{\bP^N_b})=0$ if $j<0$, and hence $\omegai=0$
  and $\omegaib=0$ if $i>N$. Since $\bP^N_b$ is smooth and thus all the local rings
  are regular, it also follows that
  $\sExt^{j}_{\bP^N_b}(\sO_{X_b},\omega_{\bP^N_b})=0$ if $j>N$, and hence
  $\omegaib=0$ if $i<0$.

  Next, consider the restriction map \cite[1.8]{MR555258},
  $$
  \varrhoib: \sExt^{N-i}_{\bP^N_B}(\sO_X,\omega_\pi)\resto{X_b} \to
  \sExt^{N-i}_{\bP^N_b}(\sO_{X_b},\omega_{\bP^N_b}).
  $$
  We have just observed that the target of the map is $0$ if $i<0$. In particular,
  $\varrhoib$ is surjective in that range. Then by \cite[1.9]{MR555258} $\varrhoib$ is an
  isomorphism and therefore $\omegai=0$ if $i<0$.
\end{proof}

\begin{lem}\label{item:2}
  Let $\sF$ be a coherent sheaf on $X$, \ujj{$i\in\bN$}, and assume that
  $\myR^i\pi_*(\sF(-q))$ is locally free for $q\gg 0$.  Then
  % \begin{enumerate}
  % \item 
  $$
  \qquad
  \pi_*\sExt^{N-i}_{\bP^N_B}(\sF, \omega_{\pi}(q))\simeq
  \sHom_B(\myR^i\pi_*(\sF({-q})), \sO_B) \qquad\text{ for $q\gg 0$}.
  $$
  % , and
  % \item $\sExt^{N-i}_{\bP^N}(\sF, \omega_{\bP^N}(q))=0$ for $q\gg 0$ if and only if
  %   $H^i(X,\sF(-q))=0$ for $q\gg 0$.
  % \end{enumerate}
\end{lem}

\begin{demo*}{\it Proof}
  %% Observe that $\supp \sExt^{N-i}_{\bP^N}(\sF, \omega_{\bP^N}(q))\subseteq \supp
  %% \sF \subseteq X$.   
  Let $q\gg 0$ and $U\subseteq B$ an affine open set such that
  $\myR^i\pi_*(\sF(-q))\resto{U}$ is free.  Then by \cite[III.6.7]{Hartshorne77} and
  \cite[II\ujjj{I}.5.2]{MR0222093},
  \begin{multline*}
    H^0(\pi^{-1}(U), \sExt^{N-i}_{\bP^N_B}(\sF, \omega_{\pi}(q)))\simeq
    \Ext^{N-i}_{\bP^N_U}(\sF_U(-q), \omega_{\pi_U})\simeq \\
    % \text{(Serre duality)}\qquad 
    \simeq \Hom_U(\myR^i\pi_*\sF(-q)\resto U, \sO_U)
    \simeq H^0(U, \sHom_B(\myR^i\pi_*\sF(-q), \sO_B)).\qed
  \end{multline*}
  % Furthermore, as $\sExt^{N-i}_{\bP^N}(\sF, \omega_{\bP^N}(q))\simeq
  % \sExt^{N-i}_{\bP^N}(\sF, \omega_{\bP^N})(q)$, we may assume that it
  %    % $\sExt^{N-i}_{\bP^N}(\sF, \omega_{\bP^N}(q))$ 
  % is generated by global sections.  Therefore, we obtain that
  % $\sExt^{N-i}_{\bP^N}(\sF, \omega_{\bP^N}(q))=0$ if and only if $H^0(\bP^N,
  % \sExt^{N-i}_{\bP^N}(\sF, \omega_{\bP^N}(q)))=H^i(X, \sF\otimes \sL^{-q})=0$.
\end{demo*}

\noindent
\ujjj{%
  The following statement and its consequences will be needed in the proof of
  \eqref{thm:coh-base-change}.  It is likely known to experts, but we could not find
  an appropriate reference.  }

\ujj{%
  \begin{lem}\label{lem:meta-DB-cover}
    Let $Z$ be a complex scheme of finite type and $\phi_\bdot:Z_\bdot\to Z$ a
    hyperresolution.  Let $\pi:W\to Z$ a morphism such that $\psi_\bdot: W_\bdot
    \leteq W\times_Z Z_\bdot\to W$ is also a hyperresolution. Let $\pi_i:W_i\to Z_i$
    be the morphisms induced by $\pi$ and assume that the natural transformation
    $\myL\pi^*\myR{\phi_\bdot}_*\to\myR{\psi_\bdot}_*\myL\pi_\bdot^*$ induces an
    isomorphism 
    $$
    \myL\pi^*\myR{\phi_\bdot}_*\sO_{Z_\bdot}\simeq
    \myR{\psi_\bdot}_*\myL\pi_\bdot^*\sO_{Z_\bdot}. 
    $$
    Then
    $$\myL\pi^*\Om_Z^0\simeq \Om^0_W.$$
    In particular, if $Z$ has only DB singularities then $W$ has only DB
    singularities.
  \end{lem}

  \medskip

  \begin{subrem}
    See \cite{DuBois81,GNPP88, MR2393625, Kovacs-Schwede09} for details on
    hyperresolutions.
  \end{subrem}

  \medskip

  \begin{cor}\label{cor:smooth-fiberspace-is-DB}
    Let $Z$ be a complex scheme of finite type with only DB singularities and $\wt
    Z\to Z$ a smooth morphism.  Then $\wt Z$ has DB singularities. \qed
  \end{cor}

  \begin{cor}\label{cor:hypersurface-is-DB}
    Let $Z$ be a complex scheme of finite type with only DB singularities and
    $H\subseteq Z$ a general member of a basepoint free linear system. Then $H$ has
    DB singularities.  \qed
  \end{cor}

  \begin{cor}\label{thm:cyclic-cover-of-DB-is-DB}
    Let $Z$ be a complex scheme of finite type with only DB singularities and $\sM$ a
    semi-ample line bundle on $Z$. Let $\pi:W\to Z$ be the cyclic cover associated to
    a general section of $\sM^m$ for some $m\gg 0$ cf.\ \cite[2.50]{KM98}. Then $W$
    has only DB singularities.
  \end{cor}

  \begin{proof}
    One can easily prove that $\pi$ satisfies the conditions of
    \eqref{lem:meta-DB-cover} or argue as follows: By
    \eqref{cor:smooth-fiberspace-is-DB} the total space $M$ of $\sM$ has DB
    singularities and then the statement follows by \eqref{lem:meta-DB-cover}
    % \eqref{cor:hypersurface-is-DB}
    applied to the embedding $W\subseteq M$ \cite[9.4]{Kollar95s}.
  \end{proof}

  \begin{proof}[Proof of \eqref{lem:meta-DB-cover}] 
    The hyperresolutions $\phi_\bdot$ and $\psi_\bdot$ fit into the commutative
    diagram:
    $$
    \xymatrix{%
      Z_\bdot \ar[d]_{\phi_\bdot} &
      W_\bdot \ar[d]^{\psi_\bdot} \ar[l]_{\pi_\bdot}  \\
      Z & W \ar[l]^\pi, }
    $$
    We also obtain the following representations of the Deligne-Du~Bois complexes of
    $Z$ and $W$:
    \begin{equation*}
      % \label{eq:19}
      \Om_Z^0\simeq \myR{\phi_\bdot}_*\sO_{Z_\bdot} \quad\text{ and }\quad
      \Om_W^0\simeq \myR{\psi_\bdot}_*\sO_{W_\bdot}.
    \end{equation*}
    Then by assumption
    $$
    \myL\pi^*\Om_Z^0\simeq \myL\pi^*\myR{\phi_\bdot}_*\sO_{Z_\bdot}\simeq
    \myR{\psi_\bdot}_*\myL\pi_\bdot^*\sO_{Z_\bdot}\simeq
    \myR{\psi_\bdot}_*\sO_{W_\bdot}\simeq \Om_W^0.
    $$
  \end{proof}

}

\noindent
We will also need the base-change theorem of Du~Bois and Jarraud
\cite[Th\'eor\`eme]{MR0376678} (cf.\ \cite[4.6]{DuBois81}):

\begin{thm}\label{thm:coh-invariance}
  Let $\phi:X\to B$ be a projective {DB family}. Then $\myR^i\phi_*\sO_{X}$ is locally
  free of finite rank and compatible with arbitrary base change for all $i$. \qed
\end{thm}

\noindent
The next theorem is our main flatness and base change result.

\begin{thm}\label{thm:coh-base-change}
  Let $\phi:X\to B$ be a projective DB family \ujjj{and $\sL$ a relatively ample line
    bundle on $X$.}
  %% and use Notation~\ref{not:coh-sheaves-of-omega}. 
  Then 
  \begin{enumerate-p}
  \item\label{item:6} the sheaves $\omegai$ are flat over $B$ for all $i$, 
  \item\label{item:3} the sheaves $\phi_*(\omegai\otimes \sL^q)$ are locally free and
    compatible with arbitrary base change for all $i$ and for all $q\gg 0$, and
  \item\label{item:12} for any base change morphism $\vartheta: T\to B$ and for all $i$,
    $$
    \big(\omegai\big)_T\simeq \omegait.
    $$
  \end{enumerate-p}
\end{thm}

\begin{subrem}
%\sideremark{\uj{Changed number!!}}
  For a coherent sheaf $\sF$ on $X$, the pushforward $\phi_*\sF$ being compatible
  with arbitrary base change means that for any morphism $\vartheta:T\to B$,
  $$
  \big(\phi_*\sF\big)_T \simeq   \big(\phi_T\big)_*\sF_T.
  $$
  In particular, (\ref{thm:coh-base-change}.\ref{item:3}) implies that for any
  $\vartheta:T\to B$, 
  $$
  \big(\phi_*(\omegai\otimes \sL^q)\big)_T\simeq
  \big(\phi_T\big)_*\bigg(\big(\omegai\big)_T\otimes \sL^q_T\bigg).
  $$
  Combined with (\ref{thm:coh-base-change}.\ref{item:12}) this means that for any
  $\vartheta:T\to B$,
  \begin{equation}
    \label{eq:12}
    \big(\phi_*(\omegai\otimes \sL^q)\big)_T\simeq
    \big(\phi_T\big)_*(\omegait\otimes \sL^q_T).
  \end{equation}
  % %% Considering the meaning of $\omegai$, 
  % Arguably, (\ref{eq:12}) is the natural base change
  % statement one might hope for, but it should be noted that the analogous statement
  % does not necessarily hold on the level of the dualizing \emph{complexes}.
\end{subrem}

\begin{comment}
\begin{subrem}
%\sideremark{\uj{Changed number!!}}
  This result naturally leads to the question whether a similar statement holds for
  higher direct images of the relative dualizing sheaf. I.e., are the sheaves
  $\myR^i\phi_*\omega_\phi$ locally free of finite rank and compatible with arbitrary
  base change for all $i$\/? By \eqref{thm:coh-base-change} $\omega_\phi$
  %=h^{-n}(\clx{\omega}_\phi)$  
  is flat, but it is not obvious that its cohomologies stay constant.
\end{subrem}
\end{comment}

\begin{proof}[Proof of \eqref{thm:coh-base-change}]
  We may asssume that $B=\Spec R$ is affine.
  %
  %% Let $X\into \bP\leteq\bP^N_B$ be a projective embedding, $\pi: \bP\to B$ the
  %% structure morphism, $\sO_\pi(1)$ a $\pi$-very ample line bundle on $\bP$ and
  %% $\sL\leteq \sO_{\pi}(1)\resto X$.
  %
  By definition, $\sL^m$ is relatively generated by global sections for all
  \ujjj{$m\gg 0$}.  For a given $m\in\bN$, choose a general section $\vartheta\in
  H^0(X, \sL^m)$ and consider the $\sO_X$-algebra
  $$
  \sA_m=\bigoplus_{j=0}^{m-1}\sL^{-j}\simeq
  {\bigoplus_{j=0}^{\infty}\sL^{-j}t^j}\bigg/{\big(t^m-\vartheta\big)}
  $$
  as in \cite[2.50]{KM98}. % and \eqref{thm:cyclic-cover-of-DB-is-DB}.
  Let $Y^m\leteq \Spec_X\sA_m$ and $\sigma: Y^m\to X$ the induced finite morphism.
  Then %we have that
  $$
  \myR^i(\phi\circ\sigma)_*\sO_{Y^m}
  \simeq \myR^i\phi_*(\sigma_*\sO_{Y^m}) \simeq
  \myR^i\phi_*\sA_m \simeq \bigoplus_{j=0}^{m-1} \myR^i\phi_*\sL^{-j}
  $$
  for all $i$ and all $b\in B$. Note that by construction, this direct sum
  decomposition is compatible with arbitrary base change.
  By \uj{\eqref{thm:cyclic-cover-of-DB-is-DB}}, $\phi\circ \sigma$ is again a DB
  family and hence $\myR^i(\phi\circ\sigma)_*\sO_{Y^m}$ is locally free and
  compatible with arbitrary base change by \eqref{thm:coh-invariance}.  Since $\phi$
  is flat and $\sL$ is locally free, it follows that then $\myR^i\phi_*\sL^{-j}$ is
  locally free and compatible with arbitrary base change for all $i$ and for all
  $j>0$.
  % by Grauert's theorem \cite[III.12.9]{Hartshorne77}.
  Then taking $\sF=\sO_X$ and applying \cite[III.6.7]{Hartshorne77},
  \eqref{lem:D-duality}, and \eqref{item:2}, we obtain that
  \begin{equation}
    \label{eq:11}
    \qquad
     \phi_*(\omegai\otimes \sL^q) \simeq
     \sHom_B(\myR^i\phi_*\sL^{-q}, \sO_B)
      \qquad \text{for $q\gg 0$}.
  \end{equation}
  This proves (\ref{thm:coh-base-change}.\ref{item:3}) and then
  (\ref{thm:coh-base-change}.\ref{item:6}) follows easily by an argument similar to
  the one used to prove the equivalence of $(i)$ and $(ii)$ in the proof of
  \cite[III.9.9]{Hartshorne77}.

  To prove (\ref{thm:coh-base-change}.\ref{item:12}) we will use induction on $i$.
  Notice that it follows trivially for $i<0$ (and $i>N$, but we will not use that
  fact) by \eqref{lem:D-duality}, so the start of the induction is covered.
  Consider the pull back map,
  $$
  \varrhoit: \underbrace{\big(\sExt^{N-i}_{\bP^N_B}(\sO_X,
    \omega_\pi)\big)_T}_{\big(\omegai\big)_T} \to
  \underbrace{\sExt^{N-i}_{\bP^N_T}(\sO_{X_T}, \omega_{\bP^N_T})}_{\omegait}.
  $$
  By the inductive hypothesis $\varrhojt$ is an isomorphism and
  $\sExt^{N-j}_{\bP^N_B}(\sO_X, \omega_\pi)\simeq \omegaj$ is flat over $B$ by
  (\ref{thm:coh-base-change}.\ref{item:6}).  Then by \cite[1.9]{MR555258},
  $\varrho^{-(j+1)}_T$ is also an isomorphism. This proves
  (\ref{thm:coh-base-change}.\ref{item:12}).
\end{proof}

\begin{comment}
  $$
  \varrhoib\otimes \id_{\sO_\pi(q)}:
  \phi_*(\sExt^{N-i}_{\bP^N_B}(\sO_X,\omega_\pi(q)) \otimes\kappa(b) \to
  \sExt^{N-i}_{\bP^N_b}(\sO_{X_b},\omega_{\bP^N_b}(q)).
  $$
  : Let $q_0\in \bN$ such that $H^i(X, \omegai\otimes \sL^q)=0$ and
  (\ref{thm:coh-base-change}.\ref{item:3}) holds for all $q\geq q_0$.  Fix such a $q$
  and let $B=\cup_j U_j$ be an open cover where $U_j=\Spec R_j$ are affine subsets of
  $B$ and $H^0(X_{U_j}, \omegai\otimes \sL^q)$ is a free $R_j$-module for all $i$.
  Then the global sections of the elements of the {\v C}ech resolution of
  $\omegai\otimes \sL^q$ corresponding to the above cover are free modules
  themselves. By the choice of $q$ these free modules form a resolution of $H^0(X,
  \omegai\otimes \sL^q)$, which is then free as well.

  Now let $M_i=\oplus_{q\geq q_0} H^0(X, \omegai\otimes \sL^q)$. Then
  $\omegai\simeq \wt {M_i}$ where $\wt{M_i}$ denotes the \emph{sheaf associated} to
  $M_i$ cf.\ \cite[Def.\ on p.110, II.5.15]{Hartshorne77}. We have proved that $M_i$
  is a free $R$-module for all $i$ and since flatness is compatible with localization
  we obtain that $\omegai$ is flat over $B$.
\end{comment}

\begin{lem}
  \label{item:1} 
  Let $X$ be a subscheme of $\,\bP^N$, $\sF$ a coherent sheaf on $X$ and $\sN$ a
  fixed line bundle on $\bP^N$.
  Then $\sF$ is $S_k$ at $x$ if and only if $\sExt^j_{\bP^N}(\sF, \sN)_x=0$ for all
  $j>N-k$.
\end{lem}

\begin{proof}
  Since $\sO_{\bP^N,x}$ is a regular local ring,
  $$
  d\leteq \depth_{\sO_{X,x}}\sF_x=\depth_{\sO_{\bP^N,x}}\sF_x
  =N-\operatorname{proj}\dim_{\sO_{\bP^N,x}}\sF_x.
  $$
  Therefore, $d\geq k$ if and only if $\Ext^j_{\sO_{\bP^N,x}}(\sF_x, \sN_x)=0$ for
  all $j>N-k$.
\end{proof}

\noindent
Using our results in this section we obtain a criterion for Serre's $S_k$ condition,
analogous to \cite[5.72]{KM98}, in the relative setting.

\begin{thm}
%[(Relative Serre duality)  \protect{%\cite[III.7.6]{Hartshorne77}, \cite[5.72]{KM98}}]
  \label{thm:serre-duality} %
  Let $\phi:X\to B$ be a projective DB family, $x\in X$ and $b=\phi(x)$.  Then
  ${X_b}$ is $S_k$ at $x$ if and only if
  $$
  \omegai_x=0 \qquad\text{for $i<k$.}
  %\myR^i\phi_*\sL^{-q} =0 \qquad \text{for $i<k$ and $q\gg 0$.}
  $$
\end{thm}

\begin{proof}
  Let $\sF=\sO_{X_b}$, $j=N-i$ and $\sN= \omega_{\bP^N_b}$. Then \eqref{item:2} and
  (\ref{item:1}) imply that ${X_b}$ is $S_k$ at $x$ if and only if $\omegaib_x=0$ for
  $i<k$. Then the statement follows from (\ref{thm:coh-base-change}.\ref{item:12})
  and Nakayama's lemma.
\end{proof}

\noindent
The following result asserts the invariance of the $S_k$ property in DB families:

\begin{thm}\label{thm:local-invariance}
  %  \sideremark{S\'andor 12/15/2008: Don't we need $B$ to be reduced for this?}
  %
  Let $\phi:X\to B$ be a projective {DB family} and $U\subseteq X$ an open subset.
  Assume that $B$ is connected and 
  the general fiber $U_{b_\text{gen}}$ of $\phi\resto U$ is $S_k$. Then
  all fibers $U_b$ of $\phi\resto U$ are $S_k$.
\end{thm}

\begin{proof}
  Suppose that the fiber $U_b$ of $\phi\resto U$ is not $S_k$. Then by
  \eqref{thm:serre-duality} there exists an $i<k$ such that $\omegai_x\neq 0$ for
  some $x\in U_b$. Let $Z$ be an irreducible component of $\supp\omegai$ such that
  $Z\cap U_b\neq\emptyset$.  It follows that $Z\cap U$ is dense in $Z$.  By
  (\ref{thm:coh-base-change}.\ref{item:6}) $\omegai$ is flat over $B$ and thus $Z$
  and then also $Z\cap U$ dominate $B$. However, that implies that $Z\cap
  U_{b_\text{gen}}\neq\emptyset$ contradicting the assumption that $U_{b_\text{gen}}$
  is $S_k$ and hence the proof is complete.
\end{proof}

As mentioned in the introduction, our main application is the following.

\begin{cor}\label{cor:invariance-of-S_k-proj}
  Let $\phi:X\to B$ be a projective {family with potentially lc fibers} or a
  projective lc morphism and $U\subseteq X$ an open subset.  Assume that $B$ is
  connected and the general fiber $U_{b_\text{gen}}$ of $\phi\resto U$ is $S_k$
  (resp.\ CM). Then all fibers $U_b$ of $\phi\resto U$ are $S_k$ (resp.\ CM).
\end{cor}

\begin{proof}
  Follows directly from \eqref{thm:main-proof} and \eqref{thm:local-invariance}.
\end{proof}

The following example % of Schr\"oer
shows that the equivalent statement does not hold in mixed characteristic.

\begin{example}[(Schr\"oer)]\label{ex:schroer}
  Let $S$ be an ordinary Enriques surface in characteristic $2$ (see \cite[p.\
  77]{MR986969} for the definition of ordinary). Then $S$ is liftable to
  characteristic $0$ by \cite[1.4.1]{MR986969}.  Let $\eta:Y\to \Spec R$ be a family
  of Enriques surfaces such that the special fiber %$Y_\text{special}$
  is isomorphic to $S$ and the general fiber % $Y_\text{general}$
  is an Enriques surface of characteristic $0$.

  Let $\zeta:Z\to \Spec R$ be the family of the projectivized cones over the members
  of the family $\eta$. I.e., for any $t\in \Spec R$, $Z_t$ is the projectivized cone
  over $Y_t$.  Since $K_{Y_t}\equiv 0$ for all $t\in \Spec R$, we obtain that $\zeta$
  is both a projective family with potentially lc fibers, and a projective lc
  morphism.
  By the choice of $\eta$, the dimension of the cohomology group $H^1(Y_b,
  \sO_{Y_b})$ jumps: it is $0$ on the general fiber, and $1$ on the special fiber.
  Consequently, by \eqref{lem:CM-coh}, the general fiber of $\zeta$ is CM, but the
  special fiber is not.
\end{example}

Recall the following CM condition used in the above example:

\begin{lem}\label{lem:CM-coh}
  Let $E$ be a smooth projective variety over a field of arbitrary characteristic and
  $Z$ the cone over $E$.  Then $Z$ is CM if and only if $h^i(E,\sO_E(m))=0$ for
  $0<i<\dim E$ and $m\in \bZ$.
\end{lem}

\begin{proof}
  See \cite[Ex.~71]{Kollar_Exercises} and \cite[3.3]{Kovacs99}.
\end{proof}

The most natural statement along these lines would be if we did not have to assume
the existence of the projective compactification of the family $U\to B$. This is
related to the following conjecture, which is an interesting and natural problem on
its own:

\begin{conj}\label{conj:compacifying-lc-morphisms}
  Let $\psi:U\to B$ be an affine, finite type lc morphism. Then there exists a base
  change morphism $\vartheta:T\to B$ and a projective lc morphism $\phi:X\to T$ such
  that $U_T\subseteq X$ and $\psi_T=\phi\resto {U_T}$.
\end{conj}

We expect that \eqref{conj:compacifying-lc-morphisms} should follow from an argument
using MMP techniques but it might require parts that are at this time still open,
such as the abundance conjecture.
On the other hand, \eqref{conj:compacifying-lc-morphisms} would clearly imply the
following strengthening of \eqref{cor:invariance-of-S_k-proj}:

\begin{proclaimspecial}{\bf Conjecture-Corollary}\label{conj:inv-of-S_k}
  Let $\psi:U\to B$ be a finite type lc morphism. Assume that $B$ is connected and
  the general fiber of $\psi$ is $S_k$ (resp.\ CM). Then all fibers are $S_k$ (resp.\
  CM).
\end{proclaimspecial}

\begin{comment}
  \begin{proof}
    By \eqref{thm:main-proof} and \eqref{thm:local-invariance} the statement follows
    from \eqref{thm:compacifying-lc-morphisms}.
  \end{proof}
\end{comment}

%%% Local Variables: 
%%% mode: latex
%%% TeX-master: "LCDB"
%%% End: 

%\bibliographystyle{/home/kovacs/tex/TeX_input/myalpha}\bibliography{/home/kovacs/tex/TeX_input/ref} %$
%\end{document}

\def\cprime{$'$} \def\polhk#1{\setbox0=\hbox{#1}{\ooalign{\hidewidth
  \lower1.5ex\hbox{`}\hidewidth\crcr\unhbox0}}} \def\cprime{$'$}
  \def\cprime{$'$} \def\cprime{$'$} \def\cprime{$'$}
  \def\polhk#1{\setbox0=\hbox{#1}{\ooalign{\hidewidth
  \lower1.5ex\hbox{`}\hidewidth\crcr\unhbox0}}} \def\cdprime{$''$}
  \def\cprime{$'$} \def\cprime{$'$}
\providecommand{\bysame}{\leavevmode\hbox to3em{\hrulefill}\thinspace}
\providecommand{\MR}{\relax\ifhmode\unskip\space\fi MR}
\providecommand{\MRhref}[2]{%
  \href{http://www.ams.org/mathscinet-getitem?mr=#1}{#2}
}
\providecommand{\href}[2]{#2}

\end{document}